\documentclass[10pt]{amsart}

\usepackage{latexsym}
\usepackage{amssymb}
\usepackage{amsmath}
\usepackage{xcolor}
\usepackage{tikz}
\usepackage{graphicx}
\usepackage{mathdots}
\usepackage{nicefrac}
\usepackage{verbatim}
\usepackage{mathabx}
\usepackage{enumitem}
\usepackage{harmony}
\usepackage{scalerel}
\newcommand{\lcirc}[1]{\scaleobj{.3}{\mbox{\Kr{#1}}}}

\usepackage[italian, english]{babel}

\newtheorem{theorem}{Theorem}[section]
\newtheorem{lemma}[theorem]{Lemma}
\newtheorem{proposition}[theorem]{Proposition}
\newtheorem{corollary}[theorem]{Corollary}

\theoremstyle{definition}
\newtheorem{definition}[theorem]{Definition}

\newtheorem{example}[theorem]{Example}

\theoremstyle{remark}
\newtheorem{remark}[theorem]{Remark}

\theoremstyle{remark}
\newtheorem{notation}[theorem]{Notation}

\def\N{\mathbb{N}}

\def\U{\mathcal{U}}
\def\V{\mathcal{V}}

\def\oF\textcircled{\text{\tiny{F}}}

\begin{document}

\title
{Central sets and infinite monochromatic exponential patterns}

\author{Mauro Di Nasso}

\author{Mariaclara Ragosta}

\address{Dipartimento di Matematica\\
Universit\`a di Pisa, Largo B. Pontecorvo 5, 56127 Pisa, Italy.}
\email{mauro.di.nasso@unipi.it}

\address{Department of Applied Mathematics,
Charles University, Faculty of Mathematics and Physics
Ke Karlovu 3, 121 16 Praha 2, Czech Republic.}
\email{mariaclara@kam.mff.cuni.cz}



\begin{abstract}
A cornerstone of Arithmetic Ramsey Theory is \emph{Hindman's Theorem} of 1974:
``For every finite coloring of the natural numbers, there exists an
infinite sequence $(x_n)$ such that all finite sums $x_{n_1}+\ldots+x_{n_k}$
of distinct elements are monochromatic".
We extend the validity of Hindman's theorem to a broad class
of non-associative operations that 
generalize exponentiation between natural numbers. 
The main tool we use in our proofs is given by the central sets, a special class of sets isolated in 1981 
by H. Furstenberg \cite{fu} in the context of topological dynamics.
It was later discovered in 1990 by V. Bergelson and N. Hindman \cite{bh}
that central sets can be characterized as those sets that belong to
a minimal idempotent ultrafilter, thus opening up the study of their rich combinatorial structure
with the well-developed machinery of algebra in the space of ultrafilters.
As a corollary of our main result, we obtain an extension of 
Sahasrabudhe's results of 2018 about the existence
of arbitrarily large (but finite) monochromatic exponential patterns; 
indeed, we obtain the existence of an infinite sequence such that
all finite exponential configurations originating from its elements are
monochromatic. 
\end{abstract}

\maketitle

\section*{Introduction}

A fundamental result of Arithmetic Ramsey Theory is \emph{Hindman's Theorem} \cite{hi74} of 1974:
``For every finite coloring (\emph{i.e.}, partition) of the natural numbers, there exists an
infinite sequence $(x_n)$ such that all finite sums $x_{n_1}+\ldots+x_{n_k}$
of distinct elements are monochromatic (that is, they belong to the same piece of the partition)".
The original combinatorial proof by N. Hindman was rather lengthy and complex.\footnote
{~Quoting N. Hindman's own words: 
``Anyone with a very masochistic bent is invited to wade through
the original combinatorial proof"; ``If the reader has a graduate student 
that she wants to punish, she should make him read and understand that original proof."}
However, soon an alternative approach was found by combining
a previous result by F. Galvin with an observation of S. Glazer
about the existence of ultrafilters on $\N$ that are idempotent
with respect to an appropriate operation. The resulting proof 
was particularly elegant and short, and paved the way for
a whole field of research; 
in fact, since then N. Hindman himself, D. Strauss, V. Bergelson,
and their co-authors have produced an impressive number of relevant new results
in Arithmetic Ramsey Theory, obtained by exploiting the algebraic properties 
of the compact space of ultrafilters (see the volume \cite{hs}).
In particular, the ultrafilter approach provides a direct way to extend
the validity of the original Hindman's Theorem from $(\N,+)$ to arbitrary semigroups. 

Typically, results in Arithmetic Ramsey Theory are about \emph{finite} configurations;
a fundamental example is the celebrated van der Waerden's Theorem \cite{vdW} of 1927,
stating that in every finite coloring of the natural numbers one
finds arbitrarily long (but finite) arithmetic progressions that are monochromatic.
In contrast, the main peculiarity of Hindman's Theorem is that it deals with \emph{infinite} configurations.
It is remarkable that results about the existence of infinite monochromatic configurations 
in semigroups are really scarce; basically, they reduce to
Milliken-Taylor's Theorem \cite{mi,ta}, which is obtained by combining
Hindman's Theorem with Ramsey's Theorem, and to its recently found polynomial extensions \cite{bhw,lu}.

In this paper we extend the validity of Hindman's Theorem to a broad class of 
operations which generalize exponentiation between natural numbers.
To the authors' knowledge, this is the first contribution aimed at generalizing Hindman's Theorem
to the non-associative case.
As a corollary of our main result, we obtain an extension of 
Sahasrabudhe's results \cite{sa} of 2018 about the existence
of arbitrarily large (but finite) monochromatic exponential patterns; 
indeed, we obtain the existence of an infinite sequence such that
all finite exponential configurations originating from its elements are monochromatic. 

The results presented here lie, in a sense, at the intersection of combinatorics, 
topological dynamics, algebra and mathematical logic. 
In fact, our results fall within Ramsey's theory,
but the proofs are based on the use of central sets,
a class of sets of natural numbers that was introduced in 1981 by H. Furstenberg \cite{fu}
in the context of topological dynamics. (A central set is defined
as the set of returns of a point $x$ to a neighborhood of 
a uniformly recurring point $y$, where $x$ and $y$ are proximal.)
Ten years later, V. Belgelson and N. Hindman found that central sets can be characterized 
in terms of algebras of ultrafilters, as elements of minimal idempotent ultrafilters 
(see \cite{hi20} for a history of central sets and their applications). 
This created a bridge between topological dynamics
and algebra on the space of ultrafilters that produced several new results.
Finally, although not made explicit in this article,
there are close connections between the algebra on ultrafilters and the
hypernatural numbers of nonstandard analysis; in fact, our initial step on the 
existence of monochromatic exponential triples $a, b, b^a$ in the natural numbers
was first obtained using nonstandard methods (see \cite[\S 6.3]{dj}).

The paper is organized as follows.
In the first section we introduce the context of our research,
namely the field of Arithmetic Ramsey Theory, and we present the
problem of the existence of exponential monochromatic configurations in the natural numbers, which
was our original motivation.
In the second section we focus on natural numbers, and
present a complete statement of our main result when the considered
operation is the exponentiation $a\star b = b^a$.
In the third section we introduce the ``exponential-like" operations on commutative semirings,
we fix the notation, and provide the precise formulation of our main theorem.
We then present numerous examples of exponential-like operations and 
some of the corresponding monochromatic configurations that are obtained,
all of which appear to be new in the literature.
In the fourth section we will briefly recall some basic notions and facts
about the algebra in the space of ultrafilters $\beta S$ where $S$ is a semigroup.
In particular, we will introduce central sets as
elements of minimal idempotent ultrafilters, and state
a version of the \emph{Central Sets Theorem} for commutative semigroups, 
a crucial tool for our proof. The fifth section is devoted to the proof of the Main Theorem;
as a preliminary step, we present a simplified
proof for the particular case of exponential patterns in $\N$.
Finally, the sixth and last section contains a list of open problems 
that emerged naturally from our research.

\smallskip
\section{The context}

Let us briefly recall the scope of this research.
An area of combinatorics, usually named ``\emph{Arithmetic Ramsey Theory},"
focuses on partition properties of the following kind:
``Given a family of \emph{patterns} $\mathfrak{P}\subset\mathcal{P}(\N)$
defined by arithmetic means, and given an arbitrary finite \emph{coloring}
(partition) $\N=C_1\cup\ldots\cup C_r$ of the natural numbers,
is it always possible to find a pattern $P\in\mathfrak{P}$ which is
\emph{monochromatic}, \emph{i.e.}, such that $P\subseteq C_i$ for some $i$?"
More generally, similar results are investigated in the context of semigroups.

Historically, one of the first results in Arithmetic Ramsey Theory is considered to be 
Schur's Theorem \cite{sc} of 1916. In that paper
it was proved that for every $n$, there exist non-trivial solutions
to Fermat's equation $x^n+y^n\equiv z^n\mod p$ for all sufficiently large primes $p$.
As a preliminary combinatorial lemma towards that result,
I. Schur showed the following property:
``In every finite coloring $\N=C_1\cup\ldots\cup C_r$
there exists a monochromatic triple $\{a, b, a+b\}\subseteq C_i$".\footnote
{~Actually, Schur proved the finite version of that result:
``For every $r$ there exists $N$ such that in every $r$-coloring
$\{1,\ldots,N\}=C_1\cup\ldots\cup C_r$ there exists a monochromatic triple 
$\{a, b, a+b\}\subseteq C_i$." 
With a compactness argument, it is shown that
the finite and infinite versions follow directly from each other.}
Several decades passed before that result was generalized to sums 
of arbitrary length by V.I. Arnoutov \cite{ar}, J. Folkman, 
and J. Sanders \cite{san} independently of each other.\footnote
{~A historical account about the correct credit for this theorem can be found in 
\cite[\S 8]{so}.}

\begin{itemize}
\item
\textbf{Arnoutov-Folkman-Sanders' Theorem.}
\emph{For every finite coloring $\N=C_1\cup\ldots\cup C_r$
and for every $k$ there exist
elements $a_1<\ldots<a_k$ such that all sums of distinct $a_i$
belong to a same color $C_i$.}
\end{itemize}

An important breakthrough was finally made in 1974 by N. Hindman \cite{hi74}
who succeeded in further extending that sum property from finite sequences to infinite sequences.

\begin{itemize}
\item
\textbf{Hindman's Theorem.}
\emph{For every finite coloring $\N=C_1\cup\ldots\cup C_r$
there exists an infinite increasing sequence $(a_n)_{n\in\N}$ such that
its set of finite sums is monochromatic; that is, there exists a color $C_i$ such that}
$$\text{FS}(a_n)_{n=1}^\infty:=\left\{\sum_{i\in F}a_i\,\Big|\,
\emptyset\ne F\subset\N\ \text{finite}\,\right\}\subseteq C_i.$$
\end{itemize}

As already mentioned in the introduction, 
the original combinatorial proof by N. Hindman was rather lengthy and complex.
However, it was soon found by F. Galvin and S. Glazer 
a much shorter alternative proof, grounded on the existence of ultrafilters on $\N$ 
that are idempotent with respect to a suitable associative operation inherited
from the sum on natural numbers. That was the starting point
of a whole field of research about algebra of ultrafilters that
produced an extensive theory and plenty of new results
in Ramsey Theory and related areas (see the volume \cite{hs}).
In particular, the ultrafilter proof provided a direct way to extend
the validity of Hindman's Theorem to arbitrary semigroups. 

The first partial result aimed to extend Schur's Theorem
from sums to exponentiations of natural numbers
was obtained in 2011 by A. Sisto \cite{si}, who proved that
monochromatic exponential triples $\{a, b, b^a\}$
can be found in every bipartition $\N=C_1\cup C_2$ of the natural numbers.
In 2018 J. Sahasrabudhe \cite{sa}, with a complex and clever 
combinatorial proof,
was able to strongly extend that result,
and showed that one can indeed find plenty of finite monochromatic 
patterns defined by means of multiplications and exponentiations.
More precisely, he proved that for every coloring $\N=C_1\cup\ldots\cup C_r$ 
(with an arbitrary finite number of colors), there exist sequences 
$(a_n)_{n=1}^\ell$ of arbitrary finite length $\ell$ such that all products and exponentiations 
of its distinct elements are monochromatic, with appropriate restrictions on the order 
in which elements are considered.
Recently \cite{dr}, the authors used the ultrafilter method to find
a much shorter proof of the monochromaticity of exponential triples,
and also proved the following infinitary version:  
``For every finite coloring of $\N$ there exists an increasing sequence $(a_n)_{n\in\N}$ such that
the infinite pattern $\{a_n\mid n\in\N\}\cup\{(a_{n+1})^{a_n}\mid n\in\N\}$
is monochromatic."

In this paper we introduce the ``ex\-po\-nen\-tial-like'' operations, which 
generalize the usual exponentiation between natural numbers,
and make sense in any commutative semigroup.
(See Definition \ref{def-exponential-like} below).
Although they are neither commutative nor associative, we show 
that these operations satisfy a strong partition regularity property
that implies the existence of \emph{infinite} monochromatic configurations,
as stated in Hindman's Theorem. 
In fact, using combinatorial properties of \emph{central sets}
as obtained in the context of algebra in the space of ultrafilters,
we prove the following result:\footnote
{~See \S \ref{sec-exponential-like} for a precise statement.}

\begin{itemize}
\item
\textbf{Hindman's Theorem for exponential-like operations.}
\emph{Let $\star$ be an exponential-like operation on a commutative semigroup $(S,*)$.
For every finite coloring of $S=C_1\cup\ldots\cup C_r$ there exists an
infinite sequence $(a_n)_{n\in\N}$ such that all ``appropriate" finite iterations of $\star$-products
of its elements are monochromatic.}
\end{itemize}

The above sequence can be assumed to be 1-1 under suitable assumptions 
about the semigroup $S$ which are satisfied by the most relevant examples, 
including of course the natural numbers with sum or with product.\footnote
{~Precisely, these hypothesis are ``weak cancellativity" and ``root-finitess" (see Definitions 
\ref{def-weaklycancellative} and \ref{def-rootfinite}).}

We remark that the above class of ``appropriate" $\star$-products 
includes all elements of the form $a_{n_1}\star\ldots\star a_{n_k}$
with brackets placed in any meaningful way (recall that
$\star$ is not associative), and
where indexes $n_1<\ldots<n_k$ are taken in increasing order, as in Hindman's Theorem.
In fact, the class of monochromatic configurations considered
is actually much broader; for example, products in which the indices are repeated and
not necessarily taken in increasing order are also ``appropriate."

As a particular case when $a\star b=b^a$ is the usual exponentiation, 
we have an extension of Sahasrabudhe's results about monochromatic exponential patterns
on $\N$, in a similar fashion as Hindman's Theorem extends Arnoutov-Folkman-Sanders's Theorem.
Indeed, we obtain that in any finite coloring of the natural numbers
there exists an \emph{infinite} increasing sequence -- rather than \emph{finite} 
sequences of arbitrary length -- such that
all finite exponential configurations originating from its elements are
monochromatic. 
We do not know whether it is also possible to obtain multiplicative patterns 
in the same color, so as to achieve a complete generalization of Sahasrabudhe's results.

To illustrate the content of our Main Theorem \ref{main} in the particular case of exponentiation 
in $\N$, we state below a consequence of it.\footnote
{~See Corollary \ref{cor-exp}.}

\begin{itemize}
\item
\textbf{Monochromatic exponential towers}.
\emph{For every (fast-growing) function $g:\N\to\N$
and for every finite coloring of the natural numbers there exists
an increasing sequence $(a_n)_{n\in\N}$ such that
for all $n_1,n_2,\ldots,n_k<s<t$
where $k\le g(s-1)$, 
the following towers of exponentiations are monochromatic:\footnote{
~When $k=1$, we agree that the above property means that 
elements $a_n$ are monochromatic.}
$${a_t}^ {   {a_s}^{    {a_{n_1}}^{   {a_{n_2}}^{\iddots^{a_{n_k}}   }    }    } }  $$}
\end{itemize}

We emphasize the fact that the above indexes $n_1,n_2,\ldots,n_k$ are not required
to be arranged in decreasing order nor to be pairwise different, but only to be 
all smaller than $s$. For instance, by appropriately choosing the function $g$, 
one obtains that also exponentiations of the following form are monochromatic: 
$${a_4}^{   {a_3}^{   {a_1}^{ {a_2}}}};\ 
{a_4}^{  {a_3}^{  {a_2}^{  {a_1}^{ {a_2}}}}}$$

Moreover, we remark that also towers of exponentiations where parentheses are 
placed in any possible position that makes sense are monochromatic.
So, \emph{e.g.}, also the following numbers will be part of the
monochromatic pattern:
$$({a_4}^{{a_3}})^{   {a_1}^{ {a_2}}};\ 
{a_4}^{\left({a_3}^{  {a_2}^{  {a_2}}}\right)^{a_2}}.$$

\section{Exponential patterns in the natural numbers}\label{sec-exp}

Before considering the general case of ``exponential-like" operations
on commutative semigroups, in this section we will focus
on the particular case of exponentiation on the natural numbers $\N$.
In fact, this is the context in which our research began and from which 
all the notions we consider in this article arose.
Moreover, we believe that treating the case of natural numbers first 
will help make the content of our general results more transparent.

\smallskip
Given an infinite sequence of natural numbers $(a_n)_{n\in\N}$,
we want to define the corresponding set of ``exponential patterns", 
similar to what is done for the set of finite sums $\text{FS}(x_n)_{n=1}^\infty$
considered in Hindman's Theorem.
We remark that -- unlike addition and multiplication -- exponentiation is not a commutative operation 
and, more importantly, neither is it associative.\footnote
{~Recall that the Finite Sums property as given by
Hindman's Theorem also holds if one replaces $(\N,+)$
with any semigroup $(S,*)$ (see \cite[\S 5.2]{hs}).}
For these reasons, attention must be paid to the order in which elements are 
considered. 

\subsection{Sets of finite exponentiations}
One possible natural way of inductively itemizing 
sets of ``exponential patterns" obtained from a finite sequence
$(a_1,\ldots,a_n)$ of natural numbers is the following.
(Since we are dealing with exponentiations, we assume that every $a_i\ge 2$.)

At step 1 we start with $\text{EXP}(a_1)=\{a_1\}$, 
and at step 2 we set $\text{EXP}(a_1,a_2)=\{a_2, {a_2}^{a_1}\}$.
At the next step we consider $a_3$, along with its powers
with exponents taken from one of the previous steps 2 or 1,
along with the triple exponentiations obtained
by first picking a power of $a_3$ with an exponent taken from
the previous step 2,
and then raising it to the exponent from step 1.
So, at step 3 we obtain the following set:
{\small $$\text{EXP}(a_1,a_2,a_3)=\left\{
a_3,\, {a_3}^{a_2},\, {a_3}^{{a_2}^{a_1}},\, {a_3}^{a_1},\, 
({a_3}^{a_2})^{a_1}={a_3}^{a_2\cdot a_1},\,
\left({a_3}^{{a_2}^{a_1}}\right)^{\!a_1}\!=
{a_3}^{{a_2}^{a_1}\cdot a_1}\right\}.$$}

Note that the exponents of $a_3$ that appear in $\text{EXP}(a_1,a_2,a_3)$
are precisely given by the products $e_2\cdot e_1$ where 
$e_2\in\text{EXP}(a_1,a_2)\cup\{1\}$
and $e_1\in\text{EXP}(a_1)\cup\{1\}$.
We follow this procedure, and give the following definition.

\begin{definition}\label{def-FE}
For every finite sequence $(a_1,\ldots,a_n)$ of natural numbers, the
\emph{set of finite exponentiations} $\text{FE}(a_1,\ldots,a_n)$
is defined as the union
$$\text{FE}(a_1,\ldots,a_n):=\bigcup_{i=1}^n\text{EXP}(a_1,\ldots,a_i)$$
where we set $\text{EXP}(a_1)=\{a_1\}$ and inductively, 
$$\text{EXP}(a_1,\ldots,a_i)=
\left\{{a_i}^{e_{i-1}\cdots\, e_1}\mid e_t\in\text{EXP}(a_1,\ldots,a_t)\cup\{1\}\
\text{for}\ t=1,\ldots,i-1\right\}.$$

The set of \emph{finite exponentiations} of an infinite sequence $(a_n)_{n\in\N}$ is the union:
$$\text{FE}(a_n)_{n=1}^\infty:=\bigcup_{n=1}^\infty\text{FE}(a_1,\ldots,a_n)=
\bigcup_{n=1}^\infty\text{EXP}(a_1,\ldots,a_n).$$
\end{definition}

Note that $\text{EXP}(a_1,\ldots,a_i)\supseteq\text{EXP}(a_{t_1},\ldots,a_{t_m},a_i)$
for all $1\le t_1<\ldots<t_m<i$ (Just take exponents $e_t=1$ for 
the indexes $t\notin\{t_1<\ldots<t_m\}$.)
In consequence, one has that
$\text{FE}(a_{t_1},\ldots,a_{t_m})\subseteq\text{FE}(a_1,\ldots,a_n)$
for all $1\le t_1<\ldots<t_m\le n$.

Also note that the above definition of $\text{FE}(a_1,\ldots,a_n)$ is given for \emph{any} 
finite sequence $(a_1,\ldots,a_n)$; in particular elements $a_i$ may repeat.

\begin{remark}
The definition of $\text{FE}(a_1,\ldots,a_n)$ above seems ``almost" the same as 
the one given by J. Sahasrabudhe in \cite{sa}, but the key difference is that 
the considered elements $a_1,\ldots,a_n$ are taken in reverse order.
This is really a crucial point. In fact, as remarked at the end of \cite[\S 5]{sa}, 
an infinitary version of the monochromaticity of exponential configurations
does not hold if one takes the reverse order on the indices,
even when restricting to simple patterns such as
$\{a_n\mid n\in\N\}\cup\{{a_n}^{a_{n+1}}\mid n\in\N\}$.
We do not have a clear and precise understanding 
of why this difference is so significant. However, we would like to emphasize that, 
when dealing with infinite configurations, the two operations are drastically different.
In fact, let's focus on how arbitrarily high towers of exponentiations can be obtained.
If we follow Sahasrabudhe's approach, it is possible to fix a base and build increasingly 
higher towers by adding elements on top with increasingly larger indexes. 
\[a_1,\; a_1^{a_2}, a_1^{a_2^{a_3}}, \ldots, \; a_1^{a_2^{\iddots^{a_7^{a_8}}}}, \ldots\]
Instead, according to our definition, in order to obtain increasingly higher towers, 
at a certain point it is necessary to change the base; indeed, a tower 
of height $k$ can only be constructed if the base has an index greater than or equal to $k$.
\[a_1, \;a_2^{a_1}, a_3^{a_2^{a_1}}, \ldots, a_8^{a_7^{\iddots^{a_3^{a_2^{a_1}}}}}, \ldots\]
\end{remark}

\smallskip
Another possible natural definition for sets of exponentiations is the following.

\begin{definition}\label{def-TE}
For every finite sequence $(a_1,\ldots,a_n)$ of natural numbers, the
set of \emph{towers of exponentiations} $\text{TE}(a_1,\ldots,a_n)$
is defined as the union
$$\text{TE}(a_1,\ldots,a_n):=\bigcup_{i=1}^n\text{EXP}'(a_1,\ldots,a_i)$$
where we set $\text{EXP}'(a_s)=\{a_s\}$ for every $s=1,\ldots,n$,
and then, inductively for all $1\le s_1<\ldots<s_m\le n$, we set:
$$\text{EXP}'(a_{s_1},\ldots,a_{s_m})=
\bigcup_{\ell<m}\left\{y^x\mid x\in\text{EXP}'(a_{s_1},\ldots,a_{s_\ell}),\,
y\in\text{EXP}'(a_{s_{\ell+1}},\ldots,a_{s_m})\right\}.$$ 

The set of \emph{towers of exponentiations} of an infinite sequence $(a_n)_{n\in\N}$ is the union:
$$\text{TE}(a_n)_{n=1}^\infty:=\bigcup_{n=1}^\infty\text{TE}(a_1,\ldots,a_n)=
\bigcup_{n=1}^\infty\text{EXP}'(a_1,\ldots,a_n).$$
\end{definition}

\begin{remark}\label{TE}
We observe that the set $\text{TE}(a_1,\ldots,a_n)$ 
is precisely the set of all towers of exponentiations of elements 
from the sequence $(a_1,\ldots,a_n)$ where
indexes are taken in decreasing order
and parentheses are placed in any possible manner that makes sense.
\end{remark}

Let us clarify this point. For simplicity, let us denote $x\star y$ to mean $y^x$.
One usually writes
$a^{b^c}$ to mean $(c\star b)\star a$;
$a^{b^{c^d}}$ to mean $((d\star c)\star b)\star a$, and so forth.
In other words, towers of exponentiations as they are usually written 
without the use of parentheses actually correspond to iterations of exponentiations 
in which the parentheses ``accumulate'' on the left-hand side, \emph{i.e.}
$${a_{1}}^ {    {a_{2}}^{    {a_{3}}^{{\iddots^{a_{n}}   }    }    } } 
\ \text{means}\quad
(((\ldots((a_n\star a_{n-1})\star a_{n-2})\ldots)\star a_3)\star a_2)\star a_1.$$
It is easily seen that
$$\left\{\,{a_{m_1}}^ {    {a_{m_2}}^{    {{\iddots^{a_{m_k}}   }    }    } }\  \Bigg| \ 
n\ge m_1>m_2>\ldots>m_k\ge 1\,\right\}\subseteq \text{TE}(a_1,\ldots,a_n).$$

However, there is plenty of other possible correct ways of placing parentheses;
\emph{e.g.}, $a^{(b^c)^d}=a^{b^{cd}}$ corresponds to 
$(d\star(c\star b))\star a$,
$(a^b)^{(c^d)}=a^{b c^d}$ corresponds to 
$(d\star c)\star(b\star a)$,
$((a^b)^c)^d=a^{bcd}$ corresponds to 
$d\star(c\star(b\star a))$, and so forth.

\smallskip
We observe that all towers of exponentiations belongs 
to the sets $\text{FE}(a_1,\ldots,a_n)$, but not conversely.

\begin{proposition}\label{TEinclusion}
Let $(a_n)_{n\in\N}$ be a sequence of natural numbers. Then for every n:
$$\text{EXP}\,'(a_1,\ldots,a_n)\subseteq \text{EXP}(a_1,\ldots,a_n)$$
and the inclusion is strict when $n\ge 3$. Hence,
$\text{TE}(a_n)_{n=1}^\infty\subsetneq\text{FE}(a_n)_{n=1}^\infty$.
\end{proposition}

\begin{proof}
This is a particular case of Proposition \ref{TEthetainclusion} in the next section
(see Remark \ref{particularcase}).
\end{proof}

\smallskip
\subsection{Larger sets of exponential patterns}
The sets of finite exponentials considered can be further expanded as follows:

\begin{definition}\label{def-FENPhi}
Let $N\in\N$, and let $\Phi=(f_n\mid n\ge 2)$ be a sequence
of functions $f_n:\N^{n-1}\to\N$.
For every infinite sequence $(a_n)_{n\in\N}$ of natural numbers,
the corresponding set of \emph{finite $(N,\Phi)$}-\emph{exponentiations} 
is the union:
$$\text{FE}_{N,\Phi}(a_n)_{n=1}^\infty:=
\bigcup_{k=1}^\infty\text{EXP}_{N,\Phi}(a_1,\ldots,a_k)$$
where $\text{EXP}_{N,\Phi}(a_1)=\{a_1\}$, and for $k\ge 2$:
\begin{multline*}
\text{EXP}_{N,\Phi}(a_1,\ldots,a_k) =
\\
= \left\{{a_k}^{\prod_{t=1}^{k-1}{a_t}^{\lambda_t}}\ \bigg{|}\ 
0\le\lambda_1\le N,\ 0\leq \lambda_t\leq f_t(a_1,\ldots,a_{t-1})\ 
\text{for}\ 2\le t\le k-1\right\}.
\end{multline*}
\end{definition}

\smallskip
\begin{remark}
Theorem 4 of \cite{sa}, which is the most general result of that paper, 
considers sets $\text{FEP}_W(x_1,\ldots,x_n)$ of finite exponentials and finite products 
that enlarge $\text{FE}(a_1,\ldots,a_n)$ using a fixed ``weight function'' $W$ 
defined on the family of subsets of $\{1,\ldots,n\}$.
When restricted to exponentiations,
those sets $\text{FEP}_W(x_1,\ldots,x_n)$ are essentially the same
as our sets $\text{FE}_{N,\Phi}(a_1,\ldots,a_n)=
\bigcup_{k=1}^n\text{EXP}_{N,\Phi}(a_1,\ldots,a_k)$.\footnote
{~The definition of $\text{FEP}_W(x_1,\ldots,x_n)$ as displayed
on \cite[page 16]{sa} seems to be inconsistent. 
In accordance with the author's description of it immediately below, 
we assume that, in the correct formulation, 
$\text{FP}_{[m]\setminus B,W}(x_{i+1},\ldots,x_m)$
should be replaced by 
$\text{FP}_{\{i+1,\ldots,m\}\setminus B,W}(x_1,\ldots,x_m)$,
so that:
$$\text{FEP}_W(x_1,\ldots,x_m)=
\left\{\prod_{i\in B}x_i^{e_i}:\emptyset\ne B\subseteq [m],e_i\in
\text{FP}_{\{i+1,\ldots,m\}\setminus B,W}(x_1,\ldots,x_m)\ \text{for each}\ i\in B\right\}.$$
Besides, we assume that 
$\text{FP}_{\emptyset, W}(x_1,\ldots,x_m)=\{1\}$.}
Indeed, with the notation in \cite{sa}, all exponential elements in 
$\text{FEP}_W(x_1,\ldots,x_n)$ 
are obtained when $B$ is a singleton,
and they are numbers of the form ${x_i}^{\mu_i}$ for $i\le n$
where the exponents $\mu_i=\prod_{s=i+1}^n {x_s}^{\lambda_s}$
for appropriate $0\le\lambda_s\le W(\{x_{s+1},\ldots,x_n\})$.
It only takes a straightforward verification to show that, by reversing the order of elements,
all such numbers belong to $\text{FE}_{N,\Phi}(x_n,\ldots,x_1)$
where $N=W(\emptyset)$ and where $\Phi=(f_n\mid n\ge 2)$
is the sequence of functions where
$$f_{k+1}(x_1,\ldots,x_k):=W(\{x_1,\ldots,x_k\}).$$ 
Conversely, given $N\in\N$ and a sequence of functions $\Phi=(f_n)_{n\ge 2}$,
it is also checked in a straightforward manner that
every element of $\text{FE}_{N,\Phi}(a_1,\ldots,a_n)$
belongs to $\text{FEP}_W(a_n,\ldots,a_1)$ where $W$ is the
weight function 
$$W(\{x_1,\ldots,x_k\}):=
\max_{\sigma\in\mathfrak{S}_k}f_{k+1}(x_{\sigma(1)},\ldots,x_{\sigma(k)}).$$
($\mathfrak{S}_k$ denotes the set of all permutations of
$\{1,\ldots,k\}$.)
\end{remark}

\smallskip
Note that if $\Phi=(f_n\mid n\ge 2)$ and $\Psi=(g_n\mid n\ge 2)$ are two
sequences of functions where $f_n(x_1,\ldots,x_{n-1})\le g_n(x_1,\ldots,x_{n-1})$ 
for all $n,x_1,\ldots,x_{n-1}\in\N$, then clearly 
$\text{EXP}_{N,\Phi}(a_1,\ldots,a_k)\subseteq\text{EXP}_{N,\Psi}(a_1,\ldots,a_k)$
for every $k$, and hence
$\text{FE}_{N,\Phi}(a_n)_{n=1}^\infty\subseteq\text{FE}_{N,\Psi}(a_n)_{n=1}^\infty$.

\smallskip
The sets of Definition \ref{def-FENPhi}
are actually much larger than the sets of exponential patterns
seen in the previous subsection, provided that the considered functions $f_n$ are 
sufficiently ``fast-growing."

\begin{proposition}\label{FEinclusion}
If $\Phi=(f_n\mid n\ge 2)$ is the sequence of functions
$f_n:\N^{n-1}\to\N$ that are inductively defined by setting 
$$\begin{cases}
f_1(x_1)=1
\\
f_{n}(x_1,\ldots,x_{n-1})=\prod_{t=1}^{n-1}{x_{t}}^{f_t(x_1,\ldots,x_{t-1})}.
\end{cases}$$
then for every infinite sequence $(a_n)_{n\in\N}$ of natural numbers and for every $k$, one has 
$\text{EXP}(a_1,\ldots,a_k)\subseteq\text{EXP}_{1,\Phi}(a_1,\ldots,a_k)$, and hence 
$\text{FE}\,(a_n)_{n=1}^\infty\subseteq\text{FE}_{1,\Phi}(a_n)_{n=1}^\infty$.
\end{proposition}

\begin{proof}
This is a particular case of Proposition \ref{FEthetainclusion} in the next section 
(see Remark \ref{particularcase2}).
\end{proof}

\begin{example}\label{example-towers3}
Let $\Phi=(f_n\mid n\ge 2)$ where 
$f_2(x)={x}^{x^x}$.
In the definition of $\text{EXP}_{N,\Phi}(a_1,a_2,a_3)$
one has $\lambda_1\le N$ and $\lambda_2\le{a_1}^{{a_1}^{a_1}}$, 
and so, \emph{e.g}, the following towers of exponentiations belong to 
$\text{EXP}_{1,\Phi}(a_1,a_2,a_3)$, and hence to $\text{FE}_{1,\Phi}(a_n)_{n=1}^\infty$:  
$${a_3}^{{a_2}^2},\ {a_3}^{{a_2}^{a_1+1}},  
\ {a_3}^{{a_1}^{\!N} {a_2}^{a_1}},
\ {a_3}^{{a_2}^{{a_1}^{2}}},
\ {a_3}^{{a_2}^{{a_1}^{a_1}}}, \ {a_3}^{{a_2}^{{a_1}^{{a_1}^{a_1}}}}.$$
Notice that in general none of the above numbers belong to $\text{FE}(a_1,a_2,a_3)$.
\end{example}

\smallskip
We are finally ready to state our main result about 
partition regularity of exponential patterns in the natural numbers
(see Corollary \ref{cor-mainexp}).

\begin{theorem}[Main Theorem about exponentiation on $\N$]\label{exp}
Let $N\in\N$ and let $\Phi=(f_n\mid n\ge 2)$
be a sequence of functions $f_n:\N^{n-1}\to\N$.
For every finite coloring $\N=C_1\cup\ldots\cup C_r$ and for every non-trivial multiplicatively
closed subset $X\subseteq\N$, there exists an increasing sequence
$(a_n)_{n\in\N}$ such that $\text{FE}_{N,\Phi}(a_n)_{n=1}^\infty\subseteq X\cap C_i$
is monochromatic and included in $X$.
\end{theorem}

So, for example, we can require that the elements 
of the monochromatic configuration
are all powers of a fixed natural number $k\ge 2$, or are sums of two squares, or are 
congruent to 1 modulo $n$, and so on.

Here is a relevant consequence of the Main Theorem.

\begin{corollary}\label{cor-exp}
Let $g:\N\to\N$ be any (fast-growing) function. Then for every finite coloring 
of $\N$ and for every non-trivial multiplicatively
closed subset $X\subseteq\N$, there exists an increasing sequence
$(a_n)_{n\in\N}$ such that for all $n_1,n_2,\ldots,n_k<s<t$ where 
$k\le g(s-1)$,
the following elements are monochromatic and belong to $X$:
$${a_t}^ {   {a_s}^{    {a_{n_1}}^{   {a_{n_2}}^{\iddots^{a_{n_k}}   }    }    } }  $$
\end{corollary}

\begin{proof}
Given $g$, let $\Phi=(f_n\mid n\ge 2)$ be the sequence of functions
$f_n:\N^{n-1}\to\N$ where 
\begin{center}
\begin{tikzpicture}
\node at (0,0) {$f_n(x_1,\ldots,x_{n-1})={x_{n-1}}^{{x_{n-1}}^{\iddots^{x_{n-1}}}}$};
\node at (4,0) {$\bigg{\}}$ \ height $g(n-1)$};
\end{tikzpicture} 
\end{center}

Note that for every increasing sequence $(a_n)_{n\in\N}$ and for
all $n_1,n_2,\ldots,n_k<s<t$
where $k\le g(s-1)$, one has that the tower
$$x:= {a_t}^ {   {a_s}^{    {a_{n_1}}^{   {a_{n_2}}^{\iddots^{a_{n_k}}   }    }    } }  \in
\ \ \text{EXP}_{1,\Phi}(a_1,\ldots,a_t).$$
To see this, notice first that, since $(a_n)_{n\in\N}$ is increasing, 
we have that $a_{n_i}\le a_{s-1}$ for every $i=1,\ldots,k$. Then $x={a_t}^{{a_s}^{\lambda_s}}$ where:
$$\lambda_s= {a_{n_1}}^{{a_{n_2}}^{\iddots^{a_{n_k}}}} \le
\ {a_{s-1}}^{{a_{s-1}}^{\iddots^{a_{s-1}}}}
<\ f_s(a_1,\ldots,a_{s-1}).$$ 
To reach the thesis, apply the Main Theorem \ref{exp} where $N=1$ and the sequence of functions $\Phi$
is the one considered above.
\end{proof}

\smallskip
We pointed out already in Remark \ref{TE} that 
if $t>s>n_1>\ldots>n_k$
then all the towers considered above 
belong to $\text{TE}(a_n)_{n=1}^\infty$,
and hence to $\text{FE}(a_n)_{n=1}^\infty$.
However, we emphasize that in Corollary \ref{cor-exp}, indexes $n_1, n_2, \ldots, n_k$
are not assumed to be arranged in decreasing order, 
nor to be pairwise different; the only condition is that they are all smaller than $s$,
and that the number $k$ of indexes is bounded by the value $g(s-1)$.
For instance, for appropriate $g$, the following towers of exponentiations
can be considered, although 
they do not belong to $\text{FE}(a_n)_{n=1}^\infty$:
$${a_4}^{ {a_3}^{ {a_1}^{ {a_2}^{a_2}}}}\,;\ {a_5}^{ {a_4}^{ {a_1}^{ {a_1}^{ {a_2}^{a_3}}}}}$$

\smallskip
\section{Exponential-like operations}\label{sec-exponential-like}

In the previous section we saw that the exponentiation on natural numbers
satisfies the same infinitary partition regularity property as stated in Hindman's Theorem
(see Theorem \ref{exp}). In this section we show that that property
actually holds for a larger class of operations on commutative semigroups
that resemble the exponentiation.

\subsection{The Main Theorem}
We observe that, although \emph{non}-associative and also \emph{non}-commutative, 
exponentiation $a\star b=b^a$ is actually closely related to the
associative and commutative operation of product on $\N$.
Indeed, exponentiation is defined as an iterated product:
$$a\star b\ =\ b^a\ =\ \underbrace{\, b\cdot\,\ldots\,\cdot b\,}_{a \textrm{\,times}}.$$
In consequence, the operation of exponentiation ``distributes" with respect to product:
$(c^b)^a=c^{ab}$, \emph{i.e.} $a\star(b\star c)=(a\cdot b)\star c$.
This suggests to consider the following generalization:

\begin{definition}\label{def-exponential-like}
Let $\theta:(S,*)\to(\N,\cdot)$ be a homomorphism from a commutative semigroup $(S,*)$ 
to the multiplicative semigroup of the natural numbers.
The \emph{exponential-like} operation $*_\theta$ on $S$
induced by $\theta$ is defined by setting for all $a,b\in S$:
$$a*_\theta b\ =\ \underbrace{\, b*\,\ldots\,*b\,}_{\theta(a)\textrm{\,times}}.$$
\end{definition}

Note that if $\varepsilon$ is identity of $(S,*)$, then
$\varepsilon *_\theta x=x$ for all $x$, since 
$\theta$ is a homomorphism and hence $\theta(\varepsilon)=1$.

\begin{notation}\label{additivenotation}
For simplicity, in the following in several occasions we will adopt the additive notation for
the commutative semigroup $(S,*)$ and agree on the following.
\begin{itemize}
\item
For $x_1,\ldots,x_n\in S$, we write
$\sum_{i=1}^n x_i$ to denote the product $x_1*\ldots * x_n$.
\end{itemize}
\begin{itemize}
\item
For $\lambda\in\N$ and $x\in S$,
we write $\lambda x$ to denote the $\lambda$-th $*$-power of $x$, that is:
$$\lambda x\ =\ \underbrace{\,x*\ldots*x\,}_{\lambda\ \text{times}}.$$
\end{itemize}

The latter notation is extended to $\lambda=0$, in which case
we agree that $0 x=\varepsilon$ is the identity of $S$.
\end{notation}

We observe that for all $\lambda,\mu\in\N\cup\{0\}$ and $x\in S$:
\begin{itemize}
\item
$\lambda(\mu x)=(\lambda\cdot\mu)x$ 
where $\cdot$ denotes multiplication between natural numbers.
\end{itemize}

Note also that, according to the above notation:
\begin{itemize}
\item
$a*_\theta b = \theta(a) b$.
\end{itemize}

Similarly to exponentiation on natural numbers,
exponential-like operations distributes with respect to
the corresponding semigroup operation.

\begin{proposition}\label{distributivity}
Let $(S,*)$ be a commutative semigroup, let
$\theta:(S,*)\to(\N,\,\cdot\,)$ be a homomorphism,
and let $*_\theta$ be the induced exponential-like operation on $S$.
Then for all $a_1,a_2, b\in S$ we have
$$a_1*_\theta(a_2*_\theta b)\ =\ (a_1*a_2)*_\theta b.$$
\end{proposition}

\begin{proof}
By the definitions, and since $\theta$ is a homomorphism, we have
$a_1*_\theta(a_2*_\theta b)=
a_1*_\theta\theta(a_2) b=\theta(a_1)(\theta(a_2) b)=
(\theta(a_1)\cdot\theta(a_2))b=\theta(a_1*a_2) b=(a_1*a_2)*_\theta b$.
\end{proof}

Clearly, the basic example of an exponential-like operation
is the usual exponentiation between natural numbers.

\begin{example}
Consider the homo\-morphism $\theta:(\N,\,\cdot\,)\to(\N,\,\cdot\,)$
given by the identity $\theta:a\mapsto a$.
Then the corresponding exponential-like operation $\cdot_\theta$ on $\N$ is 
the usual exponentiation:
$$a \cdot_\theta b\ =\,\underbrace{\, b\cdot\,\ldots\,\cdot b\,}_{\theta(a)=a\textrm{\,times}}\,=\ b^a.$$
\end{example}

A relevant example of an exponential-like operation
related to the sum of $\N$ is the following.
(Several other examples will be given in Subsection \S\ref{sec-examples}.)

\begin{example}\label{example-explike}
Fix any integer $m\ge 2$, and let $\theta:(\N,+)\to(\N,\,\cdot\,)$ be the homomorphism 
$\theta:a\mapsto m^a$.
Then the corresponding exponential-like operation $+_\theta$ on $\N$ is given by:
$$a +_\theta b\ =\,\underbrace{\, b+\,\ldots\,+b\,}_{\theta(a)=m^a\textrm{\,times}}\,=\ m^a\cdot b.$$
\end{example}

\smallskip
In the following, $(S,*)$ will always be a commutative semigroup with $\varepsilon$
as its identity, $\theta:(S,*)\to(\N,\,\cdot\,)$ a homomorphism,
and $*_\theta$ be the induced exponential-like operation on $S$.

Given a sequence $(a_n)_{n\in\N}$ in $S$, we intend to define its set of ``finite $\theta$-expo\-nentiations"
similarly to how the set of finite sums $\text{FS}(a_n)_{n=1}^\infty$ is
defined in Hindman's Theorem.
Since the $*_\theta$-operation is neither commutative nor associative,
it is necessary to take into account the order in which the elements and the $*_\theta$-products 
are considered.
  
\smallskip
One possible natural way is to proceed as done in \S 1 with the exponentation
between natural numbers. 
(To avoid trivial $\theta$-exponentiations, we assume every $a_i\ne\varepsilon$.)

\smallskip
At step 1, let $\text{EXP}_\theta(a_1)=\{a_1\}$, 
and at step 2 set $\text{EXP}_\theta(a_1,a_2)=\{a_2, a_1*_\theta a_2\}$.
At the next step 3, we consider $a_3$, along with $* _\theta$-products
of elements from one of the previous steps 2 or 1 with $a_3$,
along with the triple $*_\theta$-products obtained
by taking the product of the element from step 1 with the product of
an element from step 2 and $a_3$.
So, we obtain the following set:
\begin{multline*}
\text{EXP}_\theta(a_1,a_2,a_3)=\Big{\{}
a_3,\, a_2*_\theta a_3,\, (a_1*_\theta a_2)*_\theta a_3,\, a_1*_\theta a_3,\, 
\\
a_1*_\theta(a_2*_\theta a_3),\,
a_1*_\theta((a_1*_\theta a_2)*_\theta a_3)\Big{\}}.
\end{multline*}

Note that, by using the distributivity property of Proposition \ref{distributivity},
we have that $a_1*_\theta(a_2*_\theta a_3)=(a_1*a_2)*_\theta a_3$ and
$a_1*_\theta((a_1*_\theta a_2)*_\theta a_3)=(a_1*(a_1*_\theta a_2))*_\theta a_3$.
So, every element in $\text{EXP}_\theta(a_1,a_2,a_3)$ can
be written in the form $(x_1*x_2)*_\theta a_3$ where
$x_1\in\text{EXP}_\theta(a_1)$ or $x_1=\varepsilon$,
and $x_2\in \text{EXP}_\theta(a_1,a_2)$ or $x_2=\varepsilon$.

This construction when iterated is formalized by the following

\smallskip
\begin{definition}\label{def-FEtheta}
For every sequence $(a_n)_{n\in\N}$ in $S$, 
the set of its \emph{finite $\theta$-exponentiations} is defined as the union:
$$\text{FE}_\theta(a_n)_{n=1}^\infty:=
\bigcup_{k=1}^\infty\text{EXP}_\theta(a_1,\ldots,a_k)$$
where $\text{EXP}_\theta(a_1)=\{a_1\}$, and inductively for $k\ge 2$:
\begin{multline*}
\text{EXP}_\theta(a_1,\ldots,a_k)=
\\
\left\{(x_1*\ldots * x_{k-1})*_\theta a_k\mid 
x_t\in\text{EXP}_\theta(a_1,\ldots,a_t)\cup\{\varepsilon\} 
\ \text{for}\ t=1,\ldots,k-1\right\}.
\end{multline*}
\end{definition}

Note that elements $x_t$ are allowed to be equal to the identity $\varepsilon$; 
this is necessary to include products in which only part of the previous steps 
$\text{EXP}_\theta(a_1,\ldots,a_i)$ are considered.
In particular, $a_k\in\text{EXP}_\theta(a_1,\ldots,a_k)$ because
$(\varepsilon*\ldots * \varepsilon)*_\theta a_k=\varepsilon*_\theta a_k=
\theta(\varepsilon)a_k=1 a_k=a_k$.

\smallskip
Another natural way of itemizing iterated $\theta$-exponentiations
is the analogue of the towers of exponentiations considered in \S\ref{sec-exp}.

\smallskip
\begin{definition}\label{def-TEtheta}
For every sequence $(a_n)_{n\in\N}$ in $S$, 
the set of its \emph{towers of $\theta$-exponentiations} is defined as the union:
$$\text{TE}_\theta(a_n)_{n=1}^\infty:=
\bigcup_{k=1}^\infty\text{EXP}'_\theta(a_1,\ldots,a_k)$$
where $\text{EXP}'_\theta(a_s)=\{a_s\}$ for every $s$,
and where inductively, for all $s_1<\ldots<s_m$,
\begin{multline*}
\text{EXP}'_\theta(a_{s_1},\ldots,a_{s_m}) =
\\
= \bigcup_{\ell<m}\left\{x*_\theta y\mid x\in\text{EXP}_\theta'(a_{s_1},\ldots,a_{s_\ell}),\,
y\in\text{EXP}'_\theta(a_{s_{\ell+1}},\ldots,a_{s_m})\right\}.
\end{multline*}
\end{definition}

\begin{remark}
It can be verified in a straightforward manner that
the set $\text{TE}_\theta(a_n)_{n=1}^\infty$ 
is precisely the set of all finite iterations of the $*_\theta$-operation
applied to elements from the sequence $(a_n)_{n\in\N}$ when
indexes are taken in decreasing order,
and where parentheses are placed in any possible manner that makes sense.
\end{remark}

\smallskip
We observe that all towers of exponentiations belongs to the set
of finite $\theta$-exponentiations, but not conversely.

\begin{proposition}\label{TEthetainclusion}
Let $(a_n)_{n\in\N}$ be a sequence in $S$. Then for every n:
$$\text{EXP}\,'_\theta(a_1,\ldots,a_n)\subseteq \text{EXP}_\theta(a_1,\ldots,a_n)$$
and the inclusion is strict when $n\ge 3$. Therefore,
$\text{TE}_\theta(a_n)_{n=1}^\infty\subsetneq\text{FE}_\theta(a_n)_{n=1}^\infty$.
\end{proposition}

\begin{proof}
We proceed by induction on $m$ and show that
$$\text{EXP}'_\theta(a_{s_1},\ldots,a_{s_m})\subseteq\text{EXP}_\theta(a_{s_1},\ldots,a_{s_m})$$
for all $1\le s_1<\ldots<s_m\le n$. 
The base $m=1$ is trivial from the definitions. At the inductive step, pick
$x\in\text{EXP}'_\theta(a_{s_1},\ldots,a_{s_\ell})$ and 
$y\in\text{EXP}'_\theta(a_{s_{\ell+1}},\ldots,a_{s_m})$
for some $\ell<m$. By the inductive hypothesis, 
$x\in\text{EXP}_\theta(a_{s_1},\ldots,a_{s_\ell})$ and 
$y\in\text{EXP}_\theta(a_{s_{\ell+1}},\ldots,a_{s_m})$; then
we can write $y=(y'_{\ell+1}*\ldots *y'_{m-1})*_\theta a_{s_m}$
where $y'_t\in\text{EXP}_\theta(a_{s_{\ell+1}},\ldots,a_{s_t})\cup\{\varepsilon\}$.
By distributivity,
$$x*_\theta y=x*_\theta((y'_{\ell+1}*\ldots *y'_{m-1})*_\theta a_{s_m})=
(x*y'_{\ell+1}*\ldots *y'_{m-1})*_\theta a_{s_m}$$
and we can conclude that $x*_\theta y\in\text{EXP}_\theta(a_{s_1},\ldots,a_{s_\ell})$.

As for the strict inclusion, note that the element
$a_1*_\theta((a_1*_\theta a_2)*_\theta a_3)$ belongs to $\text{EXP}_\theta(a_1,a_2,a_3)$
but is not a $\theta$-tower of exponentiations, since the
element $a_1$ is repeated.
(Recall that we are assuming every $a_i\ne\varepsilon$.)
\end{proof}

\begin{remark}\label{particularcase}
In the particular case when $\theta:(\N,\,\cdot\,)\to(\N,\,\cdot\,)$ is the identity
and the induced operation $*_\theta$ is the exponentiation on the natural numbers,
it is readily seen that
the sets $\text{EXP}_\theta(a_1,\ldots,a_n)$ and $\text{FE}_\theta(a_n)_{n=1}^\infty$
of Definition \ref{def-FEtheta} are the same as the sets
$\text{EXP}(a_1,\ldots,a_n)$ and $\text{FE}(a_n)_{n=1}^\infty$ of Definition \ref{def-FE}.
Likewise, the sets $\text{EXP}\,'_\theta(a_1,\ldots,a_n)$ and $\text{TE}_\theta(a_n)_{n=1}^\infty$
of Definition \ref{def-TEtheta} are the same as the sets
$\text{EXP}'(a_1,\ldots,a_n)$ and $\text{TE}(a_n)_{n=1}^\infty$ of Definition \ref{def-TE}.
So, Proposition \ref{TEinclusion} is just a particular case of the above Proposition \ref{TEthetainclusion}.
\end{remark}

The following is a useful characterization of the sets of finite $\theta$-exponentiations.

\begin{lemma}\label{Lambda}
Let $(a_n)_{n\in\N}$ be a sequence in $S$, and inductively define:
\begin{itemize}
\item
$\Lambda_1=\{1\}$;
\item
$\Lambda_k=\left\{\prod_{t=1}^{k-1}\theta(a_t)^{\lambda_t}\,\big{|}\,
\lambda_t\in\Lambda_t\cup\{0\},\, t=1,\ldots,k-1\right\}$. 
\end{itemize}
Then for every $k\in\N$:
\begin{multline*}
\text{EXP}_\theta(a_1,\ldots,a_k)=
\\
\left\{\left(\sum_{t=1}^{k-1}\lambda_t a_t\right)*_\theta a_k\ \bigg{|}\,\ 
\lambda_t\in\Lambda_t\cup\{0\},\,t=1,\ldots,k-1\right\}=
\{\lambda_k a_k\mid \lambda_k\in\Lambda_k\}.
\end{multline*}
\end{lemma}

\begin{proof}
We proceed by induction on $k$. The base step $k=1$ is trivially satisfied, so let $k\ge 2$.
By the definition, $z\in\text{EXP}_\theta(a_1,\ldots,a_k)$
if and only if $z=(x_1*\ldots * x_{k-1})*_\theta a_k$
for appropriate elements $x_t\in\text{EXP}_\theta(a_1,\ldots,a_t)\cup\{\varepsilon\}$.
By the inductive hypothesis, for every $t=1,\ldots,k-1$,
we have $x_t=\lambda_t a_t$ 
for an appropriate $\lambda_t\in\Lambda_t\cup\{0\}$, and so
$z=(\sum_{t=1}^{k-1}\lambda_t a_t)*_\theta a_k$ as desired.\footnote
{~When $k=1$, we agree that the empty sum $\sum_{t=1}^{k-1}\lambda_t a_t=\varepsilon$ 
is the identity of $S$. Recall also that for every $a\in S$, we have agreed that
$0 a=\varepsilon$ is the identity of $S$.}
Finally, note that, since $\theta:(S,*)\to(\N,\,\cdot\,)$ is a homomorphism, 
by definition of $a*_\theta b=\theta(a) b$ we have
$$\left(\sum_{t=1}^{k-1}\lambda_t a_t\right)*_\theta a_k\ =\ 
\theta\!\left(\sum_{t=1}^{k-1}\lambda_t a_t\right) a_k\ =\ 
\left(\prod_{t=1}^{k-1}\theta(a_t)^{\lambda_t}\right) a_k\ =\ \lambda_k a_k$$
where $\lambda_k=\prod_{t=1}^{k-1}\theta(a_t)^{\lambda_t}\in\Lambda_k$.
\end{proof}

\smallskip
As seen above, the set of finite $\theta$-exponentiations
$\text{FE}_\theta(a_n)_{n=1}^\infty$ is sufficiently large to include 
all possible iterations of the $*_\theta$-operations where
elements are taken in decreasing order.
However, in our main result we will consider even larger sets.

\begin{definition}\label{def-FEthetaNPhi}
For every sequence $(a_n)_{n\in\N}$ in $S$, for every $N\in\N$,
and for every sequence $\Phi=(f_n\mid n\ge 2)$ of functions $f_n:\N^{n-1}\to\N$, 
the set $\text{FE}_{\theta,N,\Phi}(a_n)_{n=1}^\infty$
of \emph{finite $(N,\Phi)$-$\theta$-exponentiations} is defined as the union:
$$\text{FE}_{\theta,N,\Phi}(a_n)_{n=1}^\infty:=
\bigcup_{k=1}^\infty\text{EXP}_{\theta,N,\Phi}(a_1,\ldots,a_k)$$
where $\text{EXP}_{\theta,N,\Phi}(a_1)=\{a_1\}$, and for $k\ge 2$:
\begin{multline*}
\text{EXP}_{\theta,N,\Phi}(a_1,\ldots,a_k) =
\\
= \left\{\left(\sum_{t=1}^{k-1}\lambda_t a_t\right)\!*_\theta a_k\ \bigg{|}\ 
0\le\lambda_1\le N,\, 0\leq \lambda_t\leq f_t(a_1,\ldots,a_{t-1})\ 
\text{for}\ 2\le t\le k-1\right\}.
\end{multline*}
\end{definition}

\smallskip
Next, we will show that the above sets, for appropriate $\Phi$, 
are indeed larger than the set of finite $\theta$-exponentiations.

\begin{proposition}\label{FEthetainclusion}
Let $(S,*)$ be a commutative semigroup, let $\theta:(S,*)\to(\N,\,\cdot\,)$
be a homomorphism, and let $*_\theta$ be the induced exponential-like operation on $S$.
If $\Phi=(f_n\mid n\ge 2)$ is the sequence of functions
$f_n:\N^{n-1}\to\N$ that are inductively defined by setting 
$$\begin{cases}
f_1(x_1)=1
\\
f_{n}(x_1,\ldots,x_{n-1})=\prod_{t=1}^{n-1}{\theta(x_{t})}^{f_t(x_1,\ldots,x_{t-1})},
\end{cases}$$
then for every infinite sequence $(a_n)_{n\in\N}$ in $S$ and for every $k$, one has 
$\text{EXP}_\theta(a_1,\ldots,a_k)\subseteq\text{EXP}_{\theta,1,\Phi}(a_1,\ldots,a_k)$, and hence 
$\text{FE}_\theta(a_n)_{n=1}^\infty\subseteq\text{FE}_{\theta,1,\Phi}(a_n)_{n=1}^\infty$.
\end{proposition}

\begin{proof}
Given $(a_n)_{n\in\N}$, let $\mu_1=1$ and for $n\ge 2$
recursively define $\mu_k=\prod_{t=1}^{k-1}{\theta(a_t)}^{\mu_t}$.
It holds that $\mu_k=f_k(a_1,\ldots,a_{k-1})$ for every $k\ge 2$. 
In addition, notice that $\mu_k=\max\Lambda_k$, where the sets $\Lambda_k$'s are 
those of Proposition \ref{Lambda}. 
Indeed, $\mu_1=1$ is the greatest element of $\Lambda_1=\{1\}$;
and at the inductive step $k$ one applies the inductive hypothesis to get
$$\max\Lambda_k=\max\left\{\prod_{t=1}^{k-1}\theta(a_t)^{\lambda_t}\,\bigg{|}\,
\lambda_t\in\Lambda_t\cup\{0\},\, t=1,\ldots,k-1\right\}=\prod_{t=1}^{k-1}\theta(a_t)^{\mu_t}=\mu_k.$$

Now the desired inclusion 
$\text{EXP}_\theta(a_1,\ldots,a_k)\subseteq\text{EXP}_{1,\Phi,\theta}(a_1,\ldots,a_k)$
easily follows. Indeed, one uses the following characterization of
$\text{EXP}_\theta(a_1,\ldots,a_k)$ given in Lemma \ref{Lambda}, where
coefficients $\lambda_t\in\Lambda_t\cup\{0\}$, and hence they satisfy the inequality
$0\le\lambda_t\le\max\Lambda_t=\mu_t=f_t(a_1,\ldots,a_{t-1})$.
$$\text{EXP}_\theta(a_1,\ldots,a_k)=
\left\{\left(\sum_{t=1}^{k-1}\lambda_t a_t\right)*_\theta a_k\ \bigg{|}\,\ 
\lambda_t\in\Lambda_t\cup\{0\},\,t=1,\ldots,k-1\right\}.$$
\end{proof}

\begin{remark}\label{particularcase2}
In the particular case when $*_\theta$ is the exponentiation on the natural numbers,
it is easy to check that the sets $\text{EXP}_{\theta,N,\Phi}(a_1,\ldots,a_n)$ and 
$\text{FE}_{\theta,N,\Phi}(a_n)_{n=1}^\infty$
of Definition \ref{def-FEthetaNPhi} are the same as the sets
$\text{EXP}_{N,\Phi}(a_1,\ldots,a_n)$ and 
$\text{FE}_{N,\Phi}(a_n)_{n=1}^\infty$ of Definition \ref{def-FENPhi}.
So, Proposition \ref{FEinclusion} is
just a particular case of Proposition \ref{FEthetainclusion}.
\end{remark}

Our main result, in the style of Hindman's Theorem, concerns the existence of an 
infinite sequence such that all the finite $\theta$-exponentiations of its elements are monochromatic. 
Clearly, to obtain a meaningful result and avoid possible trivial cases, the sequence must be 1-1. 
To this end, the semigroup under consideration must satisfy appropriate properties.

\begin{definition}\label{def-weaklycancellative}
A commutative semigroup $(S,*)$ is \emph{weakly cancellative} if for all $a,b\in S$
the set $\{x\in S\mid a* x=b\}$ is finite.\footnote
{~As the name indicates, this is a weakening of the notion of cancellativity.
Recall that a commutative semigroup $(S,*)$ is \emph{cancellative} if $xa=xb$ implies that $a=b$. 
See Definitions 4.30 and 1.14 in \cite{hs}, where
the left- and right-hand versions of these notions are introduced and studied
in the general context of semigroups (not necessarily commutative).}
\end{definition}

\begin{definition}\label{def-rootfinite}
A semigroup $(S,*)$ is \emph{root-finite} if for every $a\in S$ and for every $k\in\N$,
the set $\{x\in S\mid x^k=a\}$ is finite.
\end{definition}

Note that there are numerous examples of semigroups that are both weakly cancellative and root-finite; 
\emph{e.g.}, $(\N,+)$, $(\N,\,\cdot\,)$, torsion-free abelian groups, finitely generated Abelian groups, \emph{etc.} 
Indeed, it is only in the context of root-finite semigroups 
that exponential-like operations can behave similarly to the usual 
exponentiation between natural numbers.

Recall that a function $f:A\to B$ is called \emph{finite-to-one} if all
fibers $f^{-1}(\{b\})=\{a\in A\mid f(a)=b\}$ are finite.
A useful observation that we will use is the following.

\begin{proposition}\label{rootfinitepreserving}
Let $\theta:(S,*)\to(S',*')$ be a finite-to-one homomorphism between commutative semigroups.
If $S'$ is weakly cancellative and root-finite then
$S$ is also weakly cancellative and root-finite.
\end{proposition}

\begin{proof}
Assume for the sake of contradiction that there exist $a,b\in S$
such that the set $\Gamma:=\{x\in S\mid a*x=b\}$ is infinite.
Since $\theta$ is finite-to-one, also the image
$\theta(\Gamma)=\{\theta(x)\mid a*x=b\}\subseteq
\{\theta(x)\mid \theta(a)*'\theta(x)=\theta(b)\}\subseteq\{y\in S'\mid \theta(a)*'y=\theta(b)\}$ is infinite,
against the weak cancellativity of $S'$.
Similarly, assume that there exist $k\in\N$ and $a\in S$ such that the
set $\Xi:=\{x\in S\mid k x=a\}$ is infinite (here we are using the additive notation).
Since $\theta$ is finite-to-one, also the set
$\theta(\Xi)=\{\theta(x)\mid kx=a\}\subseteq\{\theta(x)\mid k\theta(x)=\theta(a)\}\subseteq
\{y\in S'\mid k y=\theta(a)\}$ is infinite,
against the root-finiteness of $S'$.
\end{proof}

\smallskip
We are finally ready to state our main result.

\begin{theorem}[Main Theorem]\label{main}
Let $(S,*)$ be an infinite commutative weakly cancellative and root-finite semigroup
with identity,
let $\theta:(S,*)\to(\N,\,\cdot\,)$ be a homomorphism, and
let $*_\theta$ be the induced exponential-like operation on $S$.
Then for every finite coloring $S=C_1\cup\ldots\cup C_r$, for every $N\in\N$, and
for every sequence $\Phi=(f_n\mid n\ge 2)$ 
of functions $f_n:\N^{n-1}\to\N$,
there exists a 1-1 sequence $(a_n)_{n\in\N}$ such that
$\text{FE}_{\theta,N,\Phi}(a_n)_{n=1}^\infty\subseteq C_i$ is monochromatic.
\end{theorem}

As a corollary, we obtain 

\begin{corollary}[Main Theorem about exponentiation on $\N$]\label{cor-mainexp}
Let $N\in\N$ and let $\Phi=(f_n\mid n\ge 2)$
be a sequence of functions $f_n:\N^{n-1}\to\N$.
For every finite coloring $\N=C_1\cup\ldots\cup C_r$ and for every non-trivial
multiplicatively closed $X\subseteq\N$, there exists an increasing sequence
$(a_n)_{n\in\N}$ such that $\text{FE}_{N,\Phi}(a_n)_{n=1}^\infty\subseteq X\cap C_i$ is monochromatic.
\end{corollary}

\begin{proof}
By the hypothesis, $(X,\,\cdot\,)$ is an infinite sub-semigroup of $(\N,\,\cdot\,)$, and so the
inclusion map $\theta=\imath:(X,\,\cdot\,)\to(\N,\,\cdot\,)$ is a homo\-morphism.
The induced exponential-like operation on $X$ is the exponentiation
$a\cdot_\theta b=b^a$, and so we can apply the Main Theorem above.
Indeed, the semigroup $(\N,\,\cdot\,)$, and hence also the
sub-semigroup $(X,\,\cdot\,)$, is weakly cancellative and root-finite.
Finally, by passing to a subsequence if necessary,
we can assume the 1-1 sequence $(a_n)_{n\in\N}$ to be increasing.
\end{proof}

Another infinite monochromatic configuration on $\N$ 
provided by the Main Theorem is obtained by considering 
the exponential-like operation of Example \ref{example-explike}.

\begin{theorem}\label{2^ab}
Let $N\in\N$ and let $\Phi=(f_n\mid n\ge 2)$
be a sequence of functions $f_n:\N^{n-1}\to\N$.
For every finite coloring of $\N$ and for every $m\ge 2$
there exists an increasing sequence
$(a_n)_{n\in\N}$ such that the following infinite configuration
is monochromatic:
$$\left\{{a_k}\cdot {m}^{\,\sum_{t=1}^{k-1}\lambda_t a_t}\ \bigg{|}\ 
n\in\N,\ 0\le\lambda_1\le N,\ 0\leq \lambda_t\leq f_t(a_1,\ldots,a_{t-1})\ 
\text{for}\ 2\le t\le k-1\right\}.$$
\end{theorem}

\begin{proof}
Note that if $\theta:(\N,+)\to(\N,\,\cdot\,)$ is the homomorphism $\theta(a)=m^a$,
then the above infinite configuration is $\text{FE}_{\theta,N,\Phi}(a_n)_{n=1}^\infty$.
Since $(\N,+)$ is weakly cancellative and root-finite, we reach the thesis
by applying the Main Theorem to the exponential-like operation $a+_\theta b= m^a\cdot b$.
\end{proof}

As a consequence, we also obtain infinite monochromatic configurations from an increasing sequence
where all indexes except the first two can be taken in arbitrary order, and even repeated.
The proof is entirely similar to that of Corollary \ref{cor-exp}, and so we omit it.

\begin{corollary}\label{cor-2^ab}
Let $g:\N\to\N$ be any (fast-growing) function. Then for every finite coloring 
of $\N$ and for every $m\ge 2$ there exists an increasing sequence
$(a_n)_{n\in\N}$ such that for all indexes $n_1,n_2,\ldots,n_k<s<t$ where 
$k\le g(s-1)$, the following elements are monochromatic:
$${a_t}\cdot m^{a_s\cdot m^{a_{n_k}\cdot {m^{{}^{\iddots^{a_{n_2}\cdot m^{a_{n_1}}}}}}}}$$
\end{corollary}

\smallskip
\subsection{Examples of exponential-like operations}\label{sec-examples}

In this subsection we will review several examples of exponential-like operations
to which the Main Theorem can be applied; of course, many others can be considered.

\smallskip
The first two basic examples have been presented already.

\begin{example}\label{ex:exp}
Let $\theta:(\N,\,\cdot\,)\to(\N,\,\cdot\,)$ be the homomorphism given by the identity map.
The induced exponential-like operation on $\N$ is the usual exponentiation:
$$a*_\theta b\ =\ b^a.$$
\end{example}

\begin{example}\label{example-powerm}
Given $m\ge 2$, let $\theta:(\N,+)\to(\N,\,\cdot\,)$ be the homomorphism given by the 
map $\theta(a)=m^a$. The induced exponential-like operation on $\N$ is the following:
$$a*_\theta b\ =\ m^a\cdot b.$$
\end{example}

For simplicity, in the examples of exponential-like operations $*_\theta$ that follow,
we will present only the corresponding triples $\text{EXP}_\theta(a,b)=\{a, b, a*_\theta b\}$,
so as to give an idea of the kind of monochromatic patterns provided by our Main Theorem.

\begin{example}\label{example-mpowers}
Fix any $m\ge 2$, and let $\theta:(\N,\,\cdot\,)\to(\N,\,\cdot\,)$
be the homomorphism where $\theta(a)=a^m$. 
The induced exponential-like operation on $\N$ is 
$$a*_\theta b\ =\ {b}^{\,a^m}.$$
Then the following triple is monochromatic on $\N$:
$$\left\{a,\,b,\,{b}^{\,a^m}\right\}.$$
\end{example}

\begin{example}
For a prime $p$, consider the $p$-\emph{adic valuation} $v_p:\N\to\N_0$ where
$v_p(a)=\max\{k\in\N_0\mid p^k\mid a\}$. For all primes $p$ and $q$,
the function $\theta(a)=q^{v_p(a)}$ is a homomorphism from $(\N,\,\cdot\,)$ to $(\N,\,\cdot\,)$,
and the induced exponential-like operation on $\N$ is:
$$a*_\theta b\ =\ {b}^{\,q^{\,v_p(a)}}.$$
Then for all primes $p,q$ the following triple is monochromatic on $\N$:
$$\left\{a,\,b,\,{b}^{\,q^{\,v_p(a)}}\right\}.$$
\emph{E.g.}, when $p=2$ and $q=3$, for every odd number $b$ and for all $a\in\N_0$ and $c\in\N$ 
we have that $(2^a b)*_\theta c=c^{\,3^a}$, and so the triples of the following form where $b$ 
is odd are monochromatic on $\N$:
$$\left\{2^a b,\,c,\,{c}^{\,3^a}\right\}.$$
\end{example}

\begin{example}\label{example-Omega}
Recall the \emph{Omega function} $\Omega$, defined by setting $\Omega(n)$ 
to be the total number of prime numbers in the factorization of $n$ counted with their multiplicity.
The function $\theta(n)=2^{\Omega(n)}$ is a homomorphism 
$\theta:(\N,\,\cdot\,)\to(\N,\,\cdot\,)$, and the induced exponential-like operation is:
$$a*_\theta b={b}^{\,2^{\Omega(a)}}.$$
Then the following triple is monochromatic in $\N$:
$$\left\{ a,\,b,\, {b}^{\,2^{\Omega(a)}}\right\}.$$
\end{example}

\begin{example}\label{example-sumproduct}
The following operation is associative, commutative, and cancellative on $\N$:
$$a*b=a+b+ab.$$
More generally, for every $\ell,k\in\N$ where $k^2\equiv k\ (\text{mod}\,\ell)$,
the following operation on $\N$ is associative, commutative, and cancellative:\footnote
{~These operations and the infinite monochromatic configurations they originate
have been studied in \cite{dn}, where they are denoted $\lcirc{\raisebox{-0.3ex}{\huge$\star$}}_{\ell,k}$.}
$$a*b=\frac{1}{\ell}\left((\ell a+k)(\ell b+k)-k\right).$$
Equivalently, $a*b=c\Leftrightarrow (\ell a+k)(\ell b+k)=(\ell c+k)$.
The function $\theta(a)=\ell a+k$ is a homomorphism
$\theta:(\N,*)\to(\N,\,\cdot\,)$, and the induced exponential-like operation is:
``$a*_\theta b=c\Leftrightarrow (\ell a+k)^{\ell b+k}=\ell c+k$," that is:
$$a*_\theta b\ =\ \frac{1}{\ell}( (\ell a+k)^{\ell b+k}-k).$$
Since the homomorphism $\theta$ is 1-1, the semigroup $(S,*)$ is root-finite.

Then for all $\ell,k\in\N$ where $k^2\equiv k\ (\text{mod}\,\ell)$,
the following triple is monochromatic on $\N$:
$$\left\{a,\,b,\,\frac{1}{\ell}( (\ell a+k)^{\ell b+k}-k)\right\}.$$
\emph{E.g.}, when $\ell=k=1$, we have that the triple 
$\{a, b, (a+1)^{b+1}-1\}$ is monochromatic.
\end{example}
 
\begin{example}
The commutative semigroup of
$2\times 2$ \emph{circulant matrices} 
$$\begin{pmatrix} a & b \\ b & a \end{pmatrix}$$
with entries $a>b$ in $\N$ can be seen as a commutative semigroup 
$(S,*)$ on the ``lower diagonal" $S:=\{(a,b)\in\N^2\mid a>b\}$,
where operation $*$ corresponds to matrix multiplication: 
$$(a,b)*(c,d)=(ac+bd,ad+bc)\ \Leftrightarrow\  
\begin{pmatrix} a & b \\ b & a \end{pmatrix}\cdot\begin{pmatrix} c & d \\ d & c \end{pmatrix}=
\begin{pmatrix} ac+bd & ad+bc \\ ad+bc & ac+bd \end{pmatrix}.$$
The determinant induces the map $\theta:(a,b)\mapsto a^2-b^2$, which
is a homomorphism from $(S,*)$ to $(\N,\,\cdot\,)$.
Note that $\theta$ is finite-to-one. Since $(\N,\,\cdot\,)$ is weakly
cancellative and root-finite, by Proposition \ref{rootfinitepreserving} 
$(S,*)$ is also, and hence the Main Theorem applies.
A straightforward computation shows that
\begin{multline*}
(a,b)*_\theta(c,d)\ =\ \underbrace{(c,d)*\ldots *(c,d)}_{a^2-b^2\ \text{times}}\ =
\\
=\ \left(\sum_{t\in\mathbb{E}}{{a^2-b^2}\choose{t}}c^{a^2-b^2-t}d^t,\,
\sum_{t\in\mathbb{O}}{{a^2-b^2}\choose{t}}c^{a^2-b^2-t}d^t\right)
\end{multline*}
where $\mathbb{E}$ and $\mathbb{O}$ are the sets of even and of odd numbers
in $\{0, 1, \ldots, a^2-b^2\}$, respectively. We observe that 
$$(a,b)*_\theta(c,d)=(n,m)\ \Leftrightarrow\ 
n+m=(c+d)^{a^2-b^2}\ \text{and}\ n-m=(c-d)^{a^2-b^2}.$$
Then, by applying the function $(x,y)\mapsto(x-y,x+y)$ to the monochromatic triple
$\{(a,b), (c,d), (a,b)*_\theta(c,d)\}$ in $S$, we obtain that the triple
$$\left\{\begin{pmatrix} a-b \\ a+b\end{pmatrix},\, \begin{pmatrix} c-d \\ c+d\end{pmatrix},\,
\begin{pmatrix} (c-d)^{a^2-b^2} \\ (c+d)^{a^2-b^2} \end{pmatrix}\right\}.$$
is monochromatic in $\N^2$.
By changing variables $s=a-b$, $t=a+b$, $u=c-d$, $v=c+d$,
this means that the following triple is monochromatic in $\N^2$:
$$\left\{\begin{pmatrix} s \\ t\end{pmatrix},\, \begin{pmatrix} u \\ v\end{pmatrix},\,
\begin{pmatrix} u^{st} \\ v^{st} \end{pmatrix}\right\}.$$
So, we obtain the following general property:\footnote
{~Here we are using the following well-known property:
If a family $\mathcal{A}\subseteq\mathcal{P}(S)$ is partition regular on $S$
then for every function $f:S\to T$ the ``image" family 
$\{\{f(a)\mid a\in A\}\mid A\in\mathcal{A}\}$ is partition regular on $T$. 
In fact, given a finite partition $T=C_1\cup\ldots\cup C_r$,
consider the induced finite partition
$S=f^{-1}(C_1)\cup\ldots\cup f^{-1}(C_r)$. 
By partition regularity of $\mathcal{A}$ there exists 
$A\in\mathcal{A}$ with $A\subseteq f^{-1}(C_i)$ for a suitable $i$, and hence
$\{f(a)\mid a\in A\}\subseteq C_i$.}
\begin{itemize}
\item
\emph{For every $G:\N^2\to\N$, the following triple is monochromatic on $\N$:
$$\left\{\,G(s,t); G(u,v); G(u^{st}, v^{st})\,\right\}.$$}
\end{itemize}
\emph{E.g.}, when $G(s,t)=s\cdot t$ is the product, we obtain the 
already known exponential monochromatic triple
$\{st, uv, (uv)^{st}\}$; and when $G(s,t)=s+t$ is the sum,
we obtain the monochromatic triple $\{s+t, u+v, u^{st}+v^{st}\}$; and so on.
\end{example}

The following proposition indicates a general method
to produce exponential-like operations. The proof is straightforward and is omitted.

\begin{proposition}\label{multiplicativelyclosed}
Let $S=(s_n\mid n\in\N)$ be a 1-1 enumeration
of a multiplicatively closed subset $S\subseteq\N$. Then:
\begin{enumerate}
\item
The following operation $*$ on $\N$ is associative and commutative:
$$a*b=c\ \Leftrightarrow\ s_a\cdot s_b=s_c.$$ 
\item
If $\theta:(S,*)\to(\N,\,\cdot\,)$ is the homomorphism given by the inclusion,
the induced exponential-like operation $*_\theta$ on $\N$ is defined by:
$$a*_\theta b=c\ \Leftrightarrow\ s_c={s_b}^{\,s_a}.$$
\end{enumerate}
\end{proposition}

\begin{example}
For every $m\ge 2$, let $\text{Exp}_m=\{m^a\mid a\in\N\}$ be the set of 
powers of $m$, which is multiplicatively closed.
The map $a\mapsto m^a$ is a 1-1 enumeration of $\text{Exp}_m$, and
the corresponding associative and commutative operation 
as defined in Proposition \ref{multiplicativelyclosed} is the usual sum: 
``$a* b=c\Leftrightarrow m^a\cdot m^b=m^c\Leftrightarrow a+b=c$."
The homomorphism $\theta:(\text{Exp}_m,*)\to(\N,\,\cdot\,)$ given by the inclusion
induces the exponential-like operation
``$a*_\theta b=c\Leftrightarrow m^c=(m^b)^{m^a}$," which is
the operation already considered in Example \ref{example-powerm}, namely
``$a*_\theta b\ =\ {m^a}\cdot b.$"
\end{example}

\begin{example}
For every $m\ge 2$, let $S_m=\{a^m\mid a\in\N\}$ be the set of $m$-th powers,
which is multiplicatively closed.
Clearly $a\mapsto a^m$ is a 1-1 enumeration of $S_m$, and
the corresponding associative and commutative operation 
as defined in Proposition \ref{multiplicativelyclosed} is the usual product: 
``$a* b=c\Leftrightarrow a^m\cdot b^m=c^m\Leftrightarrow a\cdot b$."
The homomorphism $\theta:(S_m,*)\to(\N,\,\cdot\,)$ given by the inclusion
induces the exponential-like operation
``$a*_\theta b=c\Leftrightarrow c^m=(b^m)^{a^m}=(b^{a^m})^m$," which is
the operation already considered in Example \ref{example-mpowers}, namely
``$a*_\theta b\ =\ {b}^{\,a^m}.$"
\end{example}

\begin{example}
The set $S\subseteq\N$ of all hypotenuses $z$ of Pythagorean triples $(x,y,z)$
is multiplicatively closed.\footnote{
~Indeed if $z^2=x^2+y^2$ and $z'^2=x'^2+y'^2$
then $(zz')^2=(xx'-yy')^2+(xy'+x'y)^2$.}
Enumerate $S=(s_n\mid n\in\N)$ in increasing order. 
The exponential operation $*_\theta$ on $\N$ as
given by Proposition \ref{multiplicativelyclosed} is described as follows: 
if $s_a$ is the $a$-th hypotenuse and $s_b$ is the $b$-th hypotenuse in the list $(s_n\mid n\in\N)$,
then $a*_\theta b$ is the $({s_b}^{s_a})$-th hypotenuse in that list.
Then we obtain that triples of the following kind are monochromatic on $\N$:
$$\{a,\,b,\,c\}\quad\text{where}\quad s_c={s_b}^{s_a}.$$
\end{example}

\smallskip
Further examples are obtained in a similar way starting from
any other multiplicatively closed subset of $\N$.
For example, for any $m\ge 2$ one can consider the set
$S_1=\{a\mid a\equiv 1 \mod m\}$, or the set
$S_2=\{a\mid a\equiv m+1 \mod (m^2+m)\}$, or the set
$S_3=\{a\mid a\equiv m^2\mod (m^2+m)\}$, and so on.  
See \cite{dpr} for several examples along these lines.

\medskip
In the next examples, we will consider exponential-like operations 
defined on products of semigroups. We outline first
the general procedure that we will follow.

For $i=1,\ldots,k$, let $(S_i,*_i)$ be a commutative 
weakly cancellative and root-finite infinite semigroup, 
and let $\theta_i:(S_i,*_i)\to(\N,\,\cdot\,)$ be a homomorphism.
Consider the commutative semigroup $(S_1\times\ldots\times S_k,*)$ 
on the Cartesian product where the operation is defined coordinate-wise:
$$(a_1,\ldots,a_k)*(b_1,\ldots,b_k)=(a_1 *_1 b_1,\ldots, a_k *_k b_k).$$
It is easily verified that $(S_1\times\ldots\times S_k,*)$ is also weakly cancellative and
root-finite. For every nonempty $F\subseteq\{1,\ldots,k\}$, the map 
$\theta_F:(a_1,\ldots,a_k)\mapsto \prod_{i\in F}\theta_i(a_i)$
is a homomorphism from $(S_1\times\ldots\times S_k,*)$ to $(\N,\,\cdot\,)$.
The induced exponential-like operation on $S_1\times\ldots\times S_k$ is defined by:
$$\begin{pmatrix} a_1 \\ \ldots \\ a_k \end{pmatrix} *_\theta \begin{pmatrix} b_1 \\ \ldots \\ b_k \end{pmatrix}\ =\ 
\begin{pmatrix} {b_1}^{\theta_F(a_1,\ldots,a_k)} \\ \ldots \\ {b_k}^{\theta_F(a_1,\ldots,a_k)}\end{pmatrix}$$
where we denoted ${b_i}^n=\underbrace{\,b_i *_i\,\cdots\,*_i b_i\,}_{n\ \text{times}}$.
Then, by the particular case of our Main Theorem when only $\theta$-exponential triples
are considered, we obtain that the following pattern is monochromatic on $S_1\times\ldots\times S_k$:
$$\left\{\begin{pmatrix} a_1 \\ \ldots \\ a_k \end{pmatrix},\ 
\begin{pmatrix} b_1 \\ \ldots \\ b_k \end{pmatrix},\ 
\begin{pmatrix} {b_1}^{\theta_F(a_1,\ldots,a_k)} \\ \ldots \\ {b_k}^{\theta_F(a_1,\ldots,a_k)}\end{pmatrix}
\right\}$$

\medskip
The next examples are particular cases of the above construction.

\begin{example}
Let $(\N^2,*)$ be the commutative semigroup where
$(a,b)*(c,d)=(a+c,b+d)$.
\begin{itemize}
\item
The map $\theta(a,b)=2^a$ is a homomorphism
from $(\N^2,*)$ to $(\N,\,\cdot\,)$.
The induced exponential-like operation on $\N^2$ yields 
the following monochromatic triple in $\N^2$:
$$\left\{\begin{pmatrix} a \\ b\end{pmatrix},\, \begin{pmatrix} c \\ d\end{pmatrix},\,
\begin{pmatrix} 2^ac \\ 2^ad \end{pmatrix}\right\}.$$
\item
The map $\theta(a,b)=2^a 3^b$ is a homomorphism
from $(\N^2,*)$ to $(\N,\,\cdot\,)$.
The induced exponential-like operation on $\N^2$ yields 
the following monochromatic triple in $\N^2$:
$$\left\{\begin{pmatrix} a \\ b\end{pmatrix},\, \begin{pmatrix} c \\ d\end{pmatrix},\,
\begin{pmatrix} 2^a 3^b c \\ 2^a 3^b d \end{pmatrix}\right\}.$$
\end{itemize}
\end{example}

The exponential-like operations seen in the previous example
were obtained starting from the semigroup structure on $\N^2$
given by the coordinate-wise sum. 
The following examples are similar and are obtained by considering
coordinate-wise operations determined by sum and product.
Note that this yields weakly cancellative root-finite semigroups, and so the Main Theorem applies.

\begin{example}
Let $(\N^2,*)$ be the commutative semigroup where 
$(a,b)*(c,d)=(a+c,bd)$.
\begin{itemize}
\item
The map $\theta_1:(a,b)\mapsto 2^a$ is a homomorphism
from $(\N^2,*)$ to $(\N,\,\cdot\,)$, and
the induced exponential-like operation yields
the following monochromatic triple in $\N^2$:
$$\left\{\begin{pmatrix} a \\ b\end{pmatrix},\, \begin{pmatrix} c \\ d\end{pmatrix},\,
\begin{pmatrix} 2^ac \\ d^{\,2^a} \end{pmatrix}\right\}.$$
\item
The map $\theta_2:(a,b)\mapsto b$ is a homomorphism
from $(\N^2,*)$ to $(\N,\,\cdot\,)$, and the
induced exponential-like operation yields
the following monochromatic triple in $\N^2$:
$$\left\{\begin{pmatrix} a \\ b\end{pmatrix},\, \begin{pmatrix} c \\ d\end{pmatrix},\,
\begin{pmatrix} bc \\ d^b \end{pmatrix}\right\}.$$
\item
The map $\theta:(a,b)\mapsto 2^a b$ is a homomorphism
from $(\N^2,*)$ to $(\N,\,\cdot\,)$, and the induced exponential-like operation yields
the following monochromatic triple in $\N^2$:
$$\left\{\begin{pmatrix} a \\ b\end{pmatrix},\, \begin{pmatrix} c \\ d\end{pmatrix},\,
\begin{pmatrix} 2^a b c \\ d^{\,2^ab} \end{pmatrix}\right\}.$$
\end{itemize}
\end{example}

Other examples can be easily obtained in same fashion by also considering
other associative operations such as $(n,m)\mapsto nm+n+m$.

%
%

\smallskip
Note that the previous examples can be generalized to infinite products.
In fact, the following property holds (the proof is straightforward and is omitted).

\begin{proposition}\label{infiniteproducts}
Let $I$ be an infinite set.
For every $i\in I$, let $(S_i,*_i)$ be a commutative semigroup with identity $\varepsilon_i$, 
and let $\theta_i:(S_i,*_i)\to(\N,\,\cdot\,)$ be a homomorphism. Consider the set
$S_\infty$ of all $I$-sequences $\mathbf{a}=(a_i)$
where $a_i\in S_i$ for every $i$ and whose support 
$$\text{Supp}(\mathbf{a}):=\{i\in I\mid a_i\ne \varepsilon_i\}$$
is finite, and let $\star$ be the coordinate-wise operation:
$(a_i)\star(b_i)=(a_i*_ib_i)$.
Then $\theta_\infty:\mathbf{a}\mapsto\prod_i \theta_i(a_i)$ is a homomorphism
from the commutative semiring $(S_\infty,\star)$ to $(\N,\,\cdot\,)$, 
and the induced exponential-like operation $\star_{\theta_\infty}$ on $S_\infty$ is the following:
$$\mathbf{a}\star_{\theta_\infty}\mathbf{b}\ =\ \left(\prod_{i\in I}\theta_i(a_i)\right)\!\mathbf{b}$$ 
where for $\lambda\in\N$ and $\mathbf{b}\in S$ we denoted
$\lambda\mathbf{b}=(\lambda\cdot b_i)$.
Besides, if all semigroups $(S_i,*_i)$ are weakly cancellative and root-finite,
then $(S_\infty,\star)$ is also weakly cancellative and root-finite.
\end{proposition}

Let us see a relevant example.

\begin{example}
Let $(S_\infty,+)$ be the commutative semiring where $S_\infty$
is the set of all sequences $\mathbf{a}=(a_n)\in{\N_0}^\N$ of non-negative integers whose support 
$$\text{Supp}(\mathbf{a}):=\{n\in\N\mid a_n\ne 0\}$$
is finite, and where $+$ is the coordinate-wise sum.
Fix any sequence $\mathbf{x}=(x_n)\in\N^\N$ of natural numbers.
The map $\theta:(S_\infty,+)\to(\N,\,\cdot\,)$ where 
$$\theta(\mathbf{a})=\prod_{n\in\N}{x_n}^{a_n}$$
is a homomorphism, and the induced exponential-like operation 
yields the following monochromatic triple in $S$:\footnote
{~This is the particular case of Proposition \ref{infiniteproducts}
where $I=\N$, and where $\theta_n(y)={x_n}^y$ for every $n$.}
$$\left\{\mathbf{a},\ \mathbf{b},\ \left(\prod_{n\in\N}{x_n}^{a_n}\right)\!\mathbf{b}\right\}.$$

\smallskip
Now let $(p_n)$ be the increasing sequence of prime numbers,
and for every $\mathbf{a}\in S_\infty$ let $a\in\N$
be the number whose factorization is $a=\prod_n{p_n}^{a_n}$.
It is readily seen that the function $\mathbf{a}\mapsto a$ is a semigroup isomorphism 
between $(S_\infty,+)$ and $(\N,\,\cdot\,)$, whose inverse map
is given by the sequence of $p_n$-adic valuations: 
$a\mapsto(v_{p_n}(a))$. Note that for every $\lambda\in\N$
one has that $\lambda\mathbf{a}=a^\lambda$.
By applying this isomorphism, we obtain the following monochromatic triple in $\N$:
$$\left\{a,\ b,\ b^{\,\prod_{n\in\N}{x_n}^{v_{p_n}(a)}}\right\}.$$

By varying the sequence $(x_n)$, numerous examples are obtained
of monochromatic triples in $\N$.
\emph{E.g.}, if $(x_n)=(p_n)$ is the sequence of prime numbers,
we obtain the exponential triple $\{a, b, b^a\}$; and
if $(x_n)$ is the constant sequence with value $2$, we obtain
the monochromatic triple $\{a,b, b^{2^{\Omega(a)}}\}$ already seen in Example \ref{example-Omega},
and so on.
\end{example}
   
\smallskip
\section{Recalls to algebra of ultrafilters and central sets}\label{sec-central}

The main tool for the proof of the Main Theorem are the \emph{central sets}, 
a special class of sets endowed with a particularly
rich combinatorial structure.
Central sets of natural numbers were originally introduced by H. Furstenberg \cite{fu} in 1981 
in the context of topological dynamics.
Several years later, in 1990, V. Bergelson and N. Hindman \cite{bh} proved that 
the same class of sets could be also be defined in an apparently disjoint context 
as members of \emph{minimal idempotent ultrafilters}.
Indeed, minimal ultrafilters precisely correspond to uniformly recurrent points
in the dynamical system $(\beta\N, S)$ given by the space of ultrafilters on $\N$
endowed with the ``shift" map $S(\U)=1\oplus\U$.
Here we follow the ultrafilter approach, through which 
the notion of central sets is naturally defined 
for arbitrary semigroups, and many relevant 
combinatorial properties have been demonstrated.

Below, we briefly recall some of the basic notions and facts
about \emph{algebra of ultrafilters},
limiting ourselves to the minimum that is needed for our purposes.
For a comprehensive treatment of the subject, we refer the reader to the 
comprehensive monograph \cite{hs}.

\subsection{Algebra on ultrafilters}
In the following, we assume the reader to be familiar with ultrafilters.

\smallskip
Let us first recall two fundamental constructions to produce ultrafilters from given ones.
\begin{itemize}
\item
Given an ultrafilter $\U$ on a set $I$ and a function $f:I\to J$,
the \emph{image ultrafilter} of $\U$ under $f$ is the ultrafilter on $J$ defined by setting
$$f(\U):=\{A\subseteq J\mid f^{-1}(A)\in\U\}.$$
\item
Given ultrafilters $\U$ and $\V$ on a set $I$, their \emph{tensor product}
is the following ultrafilter on the Cartesian product $I\times I$:
$$\U\otimes\V:=\{X\subseteq I\times I\mid\{i\in I\mid\{j\in I\mid(i,j)\in X\}\in\V\}\in\U\}.$$
\end{itemize}

\smallskip
There is a natural way to extend the sum operation on $\N$ to
an operation defined on the set $\beta\N$ of all ultrafilters on $\N$. 
Precisely, one defines the ultrafilter $\U\oplus\V$ by setting, for every $A\subseteq\N$,
$$A\in\U\oplus\V\Leftrightarrow \{n\in\N\mid A-n\in\V\}\in\U,$$
where $A-n=\{m\in\N\mid m+n\in A\}$ is the $n$-shift of $A$.
Note that $\U\oplus\V$ is the image ultrafilter $S(\U\otimes\V)$
of the tensor product $\U\otimes\V$ under the sum function $S(n,m)=n+m$.
The operation $\oplus$ turns out to be associative, but not commutative.
By identifying each natural number $n$ with the corresponding
principal ultrafilter $\U_n=\{A\subseteq\N\mid n\in A\}$,
we observe that one has $\U_n\oplus\U_m=\U_{n+m}$, and so
the semigroup $(\beta\N,\oplus)$ is an actual extension of $(\N,+)$.

The space of ultrafilters $\beta\N$ comes with the natural topology 
where the following sets $\mathcal{O}_A$ with $A\subseteq\N$ form a base of clopen sets:
$$\mathcal{O}_A:=\{\U\in\beta\N\mid A\in\U\}.$$
The resulting topological space is a Hausdorff and compact space
where $\N$ (identified with the set of principal ultrafilters $\U_n$) is dense.
Such a topological space $\beta\N$ can be characterized 
as the \emph{Stone-\v{C}ech compactification} of the discrete space $\N$.
Endowed with the above topology, the semigroup $(\beta\N,\oplus)$
has the structure of a \emph{right compact topological semigroup}, 
\emph{i.e.}, a compact topological semigroup where, for every $\V\in\beta\N$,
the operation ``on the right" $\psi_\V:\U\mapsto\U\oplus\V$ 
is a continuous function.

The above construction grounded on the natural numbers with the sum,
can be carried through starting from any semigroup $(S,*)$ to produce 
a \emph{right compact topological semigroup} $(\beta S,\circledast)$
on the topological space of ultrafilters on $S$ where the operation
is defined as follows:
$$A\in\U\circledast\V\Leftrightarrow \{s\in S\mid\{t\in S\mid s*t\in A\}\in\V\}\in\U.$$

The following is a classical result in the theory of semigroups.

\begin{itemize}
\item
\textbf{Ellis's Lemma}. 
\emph{In every right compact topological semigroup there exist idempotent elements.}
\end{itemize}

In particular, for every semigroup $(S,*)$
there exist idempotent elements in $(\beta S,\circledast)$, that is,
ultrafilters $\U$ such that $\U\circledast\U=\U$.
That class of ultrafilters has the utmost importance in applications
to combinatorics, beginning with the following fact, first pointed out by F. Galvin
in the case of $(S,*)=(\N,+)$:\footnote
{~When F. Galvin isolated that property in 1970, 
the existence of idempotent ultrafilters in $(\beta\N,\oplus)$ was still an open problem.}

\begin{itemize}
\item
(Galvin 1970). \emph{Let $\U$ be an idempotent ultrafilter in $(\beta S,\circledast)$. 
For every set $A\in\U$
there exists an infinite sequence $(a_n)_{n\in\N}$ in $S$ such that the set of its finite products}
$$\text{FP}(a_n)_{n=1}^\infty:=
\left\{a_{n_1}*\ldots*a_{n_k}\,\Big|\, n_1<\ldots<n_k\,\right\}\subseteq A.$$
\end{itemize}

Note that one obtains a proof of \emph{Hindman's Theorem} for general semigroups
by simply picking any idempotent ultrafilter $\U$ in $(\beta S,\circledast)$, and then observing
that in any finite coloring $S=C_1\cup\ldots\cup C_r$ one of the colors $C_i\in\U$.

\begin{theorem}[Hindman for semigroups]\label{hindmansemigroups}
Let $(S,*)$ be a semigroup. For every finite coloring $S=C_1\cup\ldots\cup C_r$
there exists an infinite sequence $(a_n)_{n\in\N}$ in $S$ such that the set of its finite products
$\text{FP}(a_n)_{n=1}^\infty\subseteq C_i$ is monochromatic.
\end{theorem}

We remark that to ensure that the sequence $(a_n)_{n\in\N}$ is 1-1, 
the semigroup $(S,*)$ must satisfy additional properties 
(weak cancellativity is sufficient, see the comments at the beginning of \S \ref{sec-proof}).

\subsection{Minimal ideals and central sets}
Let $(S,*)$ be a semigroup and $(\beta S,\circledast)$ its
Stone-\v{C}ech compactification.
A set $L\subseteq\beta S$ is a \emph{left ideal} if $\U\circledast\V\in L$ for every 
$\U\in\beta S$ and for every $\V\in L$.
The notion of \emph{right ideal} is defined similarly.
A left ideal $L$ is called \emph{minimal} if it does not properly contain any left ideal.
If $L$ is a minimal left ideal and $\V\in L$ is any element of its,
then by minimality it must be $L=\psi_\V(\beta S)$, and hence
$L$ is a closed subspace. In consequence, by a
straight application of \emph{Zorn's Lemma}, it is proved that every left ideal 
includes a minimal left ideal. 
An ultrafilter $\U$ is called \emph{minimal} if it belongs
to a minimal left ideal. 

\begin{remark}
The existence of minimal ultrafilters in $(\beta S,\circledast)$ that are non-principal 
requires additional ``mild" assumptions about the semigroup $S$; 
for example, that $(S,*)$ contains no finite two-sided ideals. 
Note that this last condition holds if $S$ is weakly cancellative.
\end{remark}

Sets that belong to minimal ultrafilters
have a rich combinatorial structure; \emph{e.g.}, consider the following property:

\begin{itemize}
\item
\emph{Let $\U$ be a minimal ultrafilter of $(\beta\N,\oplus)$. 
For every set $A\in\U$
and for every $\ell$ there exists an $\ell$-term arithmetic
progression $a, a+d, \ldots, a+(\ell-1)d \in A$.}
\end{itemize}

We observe that if $L$ is a minimal left ideal 
then $L$ is both topologically and algebraically closed.
In consequence, $L$ inherits from $(\beta S,\circledast)$
the structure of a right compact topological semigroup, and one can
apply \emph{Ellis' Lemma} to $(L,\circledast)$ to get the existence
of \emph{minimal idempotent} ultrafilters.

\begin{itemize}
\item
\emph{Let $L$ be a minimal left ideal of $(\beta S,\circledast)$.
Then there exist idempotent ultra\-filters $\U\in L$.}
\end{itemize}

We are finally ready to introduce the fundamental class of sets considered in this paper.

\begin{definition}
Let $(S,*)$ be an infinite semigroup. A set $A\subseteq S$ is \emph{central} if it belongs to 
a minimal idempotent ultrafilter of $(\beta S,\circledast)$.
\end{definition}

Since they share properties of both members of idempotent ultrafilters 
and members of minimal ultrafilters, it is not surprising that central sets 
are a really strong tool in combinatorics.
For example, as already mentioned, a set $A$ belonging 
to a minimal ultrafilter in $(\beta\N,\oplus)$ contains arithmetic progressions 
of arbitrary length. 
If the minimal ultrafilter is also idempotent, \emph{i.e.},
if the set $A$ is central, then we can require that the 
common differences of the progressions also belong to $A$. 
This property is basically what we will need in the proof of our Main Theorem
(see Remark \ref{braueridempotent}).

The \emph{Central Sets Theorem} describes explicitly the 
combinatorial richness of these sets.
The first version of that theorem for central sets of natural numbers
was proved by H. Furstenberg in \cite{fu}, 
but in the following years stronger versions were shown to hold.
In particular, we will use the following version proved by
D. De, N. Hindman and D. Strauss \cite{dhs} in 2008 in the
general context of commutative semigroups.

For every set $X$, we denote by $\text{Fin}(X)$ the family of its finite subsets.
For nonempty sets $A,B\subseteq\N$ we write $A<B$ to mean that $a<b$ for all $a\in A$ and $b\in B$.
We convene that $\emptyset<A$ for every nonempty $A\subseteq\N$.

Since we are dealing with commutative semigroups, for simplicity we will
adopt the additive notation. We denote by $S^\N$ the set
of all sequences $(s_n\mid n\in\N)$ of elements of $S$.

\begin{theorem}[Central Sets Theorem -- 2008 \cite{dhs}]\label{CST}
Let $(S,+)$ be a commutative semigroup, and let
$A\subseteq S$ be a central set. Then there exist maps 
$$\alpha:\text{Fin}(S^\N)\to S\quad{and}\quad
H:\text{Fin}(S^\N)\to\text{Fin}(\N)$$
such that

\begin{itemize}
\item 
$H(F)<H(G)$ for all $F\subsetneq G$ in $\text{Fin}(S^\N)$;

\smallskip
\item 
For every increasing sequence $G_1\subsetneq G_2\subsetneq\ldots\subsetneq G_n$
in $\text{Fin}(S^\N)$ and for every choice of functions $f_i\in G_i$ for $i=1,\ldots,n$, one has
$$\sum_{i=1}^n\left(\alpha(G_i)+\!\!\sum_{t\in H(G_i)}\!\!\!\!f_i(t)\right)\in A$$
\end{itemize}
\end{theorem}

The reader interested in central sets will find much useful information, 
including interesting historical observations on the development of research on the subject, 
in N. Hindman's article \cite{hi20}. 

The following direct consequence of the Central Sets Theorem
will be the most relevant for the proof of our Main Theorem.

\begin{corollary}\label{brauerforcentralsets}
Let $(S, +)$ be a commutative semigroup, and let 
$A\subseteq S$ be a central set. 
Then for every $\ell\in\N$ there exist $a,d$ such that 
$a, d, a+d, \ldots, a+\ell d\in A$.
\end{corollary}
   
\begin{proof}
Since $A$ belongs to an idempotent ultrafilter,
there exists a function $f:\N\to S$ such that the set of finite sums
$\text{FS}(f)=\left\{\sum_{t\in F}f(t)\mid \emptyset\ne F\in\text{Fin}(\N)\right\}\subseteq A$.
Let $G_1:=\{f, 2f, \ldots, (\ell+1) f\}$ where $kf:\N\to S$ is the function
where $(kf)(n)=k f(n)$. Then define
$a:=\alpha(G_1)+\sum_{t\in H(G_1)}f(t)$ and $d:=\sum_{t\in H(G_1)}f(t)$, where
$\alpha:\text{Fin}(S^\N)\to S$ and $H:\text{Fin}(S^\N)\to\text{Fin}(\N)$ 
are the maps as given by the Central Sets Theorem. Then it is easily verified
that $a, d, a+d, \ldots, a+\ell d\in A$, as desired.
\end{proof}

\begin{remark}
In the case of natural numbers, the previous corollary states that
minimal idempotent ultrafilters are ``witnesses" not only to
Hindman's Theorem, but also to Brauer's Theorem  \cite{br},
a strengthening of van der Waerden's Theorem where
the common distance also belongs to the same color of the progression.
\begin{itemize}
\item
\textbf{Brauer's Theorem.}
\emph{Let $\N=C_1\cup\ldots\cup C_r$. Then there exists a color $C_i$
such that for every
$\ell$ there exists a pattern $a,d,a+d,\ldots,a+(\ell-1)d\in C_i$.}
\end{itemize}
\end{remark}

\section{Proof of the main theorem}\label{sec-proof}

We observe that in Hindman's Theorem \ref{hindmansemigroups} for general semigroups,
one cannot require the sequence $(a_n)_{n\in\N}$ to be 1-1.\footnote
{~See \cite{gt} where this issue is addressed.}
To avoid trivial cases,
a natural assumption is to consider semigroups $(S,*)$ 
where non-principal ultrafilter $\beta S\setminus S$ form
a sub-semigroup of the Stone-\v{C}ech compactification $(\beta S,\circledast)$.
Under this assumption, one can always take a non-principal idempotent 
ultrafilter on $S$ and ensure that the sequence $(a_n)_{n\in\N}$ constructed 
with the usual Galvin-Glazer proof is 1-1 (see \cite[Thm. 5.8]{hs}).
Note that non-principal ultrafilters form a sub-semigroup
in most relevant examples; in particular, this is the case when 
the initial semigroup $S$ is weakly left cancellative
(see \cite[\S 4.3]{hs} for details).

Also in our Main Theorem we will need an extra hypothesis
if we want to ensure the existence of a sequence $(a_n)_{n\in\N}$ which is 1-1,
namely root-finiteness.
Indeed, it is only in the context of root-finite semigroups 
that exponential-like functions can behave similarly to the usual 
exponentiation between natural numbers.

\begin{theorem}\label{mainmain}
Let $(S,*)$ be a commutative root-finite semigroup with identity, 
let $\theta:(S,*)\to(\N,\,\cdot\,)$ be a
homomorphism, and let $*_\theta$ be the induced exponential-like operation on $S$.
Suppose that $\U$ is a non-principal minimal idempotent ultrafilter in the
Stone-\v{C}ech compactification $(\beta S,\circledast)$,
and let $\V=:\Theta(\U\otimes\U)$ be the image ultrafilter
under the map $\Theta(a,b)=a*_\theta b$.
Then for every $A\in\V$, for every $N\in\N$, 
and for every sequence $\Phi=(f_n\mid n\ge 2)$ 
of functions $f_n:\N^{n-1}\to\N$, there exists a 1-1 sequence $(a_n)_{n\in\N}$ in $S$
such that $\text{FE}_{\theta,N,\Phi}(a_n)_{n=1}^\infty\subseteq A$.
\end{theorem}

\begin{remark}
Note that if $\U$ is non-principal then also $\V=\Theta(\U\otimes\U)$ is non-principal.
Indeed, assume for the sake of contradiction that
$\V$ is the principal ultrafilter determined by an element $a\in S$.
Then the preimage $\Theta^{-1}(\{a\})=\{(x,y)\in S\times S\mid x*_\theta y=a\}\in\U\otimes\U$.
Since $\U$ is non-principal, it can be shown that there exists a 1-1 sequence $(s_n\mid n\in\N)$ in $S$
such that $(s_n,s_m)\in\Theta^{-1}(\{a\})$, \emph{i.e.},
$\theta(s_n) s_m=a$ (in additive notation) for all $n<m$, against the hypothesis of $S$ root-finite.\footnote
{~This is a particular case of the following well-known property: 
Let $X\subseteq S\times S$; then $X\in\U\otimes\U$ for some non-principal ultrafilter $\U$ on $S$
if and only if there exists a 1-1 sequence $(s_n\mid n\in\N)$ such that $(s_n,s_m)\in X$ for all $n<m$.} 
\end{remark}

As a corollary of the previous result, we obtain the

\medskip
\noindent
\textbf{Main Theorem} (Theorem  \ref{main})\textbf{.}\ 
\emph{Let $(S,*)$ be an infinite commutative weakly cancellative and 
root-finite semigroup with identity,
let $\theta:(S,*)\to(\N,\,\cdot\,)$ be a homomorphism, and
let $*_\theta$ be the induced exponential-like operation on $S$.
Then for every finite coloring $S=C_1\cup\ldots\cup C_r$, for every $N\in\N$, and
for every sequence $\Phi=(f_n\mid n\ge 2)$ 
of functions $f_n:\N^{n-1}\to\N$,
there exists a 1-1 sequence $(a_n)_{n\in\N}$ such that
$\text{FE}_{\theta,N,\Phi}(a_n)_{n=1}^\infty\subseteq C_i$ is monochromatic.}

\begin{proof}[Proof of the Main Theorem]
Weak cancellativity of $(S,*)$ ensures that the space
of non-principal ultrafilters on $S$ form a sub-semigroup of
$(\beta S,\circledast)$; in fact, $S$ is weakly cancellative if and only if
the non-principal ultrafilters $\beta S\setminus S$ form a left ideal of $(\beta S,\circledast)$
(see \cite[Theorem 4.31]{hs}). Since $\beta S\setminus S$ is closed,
and hence compact, there exists a non-principal
minimal idempotent ultrafilter $\U$. Let $\V=:\Theta(\U\otimes\U)$ be the image ultrafilter
under the map $\Theta(a,b)=a*_\theta b$, and let $C_i$ be the color such that $C_i\in\V$.
Then apply the previous theorem with $A=C_i$.
\end{proof}

We observe that Theorem \ref{mainmain} actually shows a combinatorial property
of central sets.

\begin{theorem}
Let $B$ be any central set of the commutative weakly cancellative and root-finite 
infinite semigroup $(S,*)$ with identity.
Then for every homomorphism $\theta:(S,*)\to(\N,\,\cdot\,)$,
for every finite coloring of $S$,
for every $N\in\N$, and for every sequence $\Phi=(f_n\mid n\ge 2)$ 
of functions $f_n:\N^{n-1}\to\N$, there exist sequences $(b_n)_{n\in\N},(c_n)_{n\in\N}\subseteq B$ 
such that the infinite pattern
$\text{FE}_{\theta,N,\Phi}(\theta(b_n) c_n)_{n=1}^\infty$ is monochromatic,
and all its elements are of the form $\theta(x) y$ for suitable $x,y\in B$.
\end{theorem}

\begin{proof}
Pick a non-principal minimal idempotent ultrafilter $\U$ in $(\beta S,\circledast)$
with $B\in\U$, and let $\V=\Theta(\U\otimes\U)$ where $\Theta(a,b)=a*_\theta b=\theta(a) b$.
Given a finite partition $S=C_1\cup\ldots\cup C_r$, let $C _j$ 
be the color such that $C_j\in\V$. Since $B\in\U$, we have that 
$$\Theta(B\times B)=\{b*_\theta c\mid b,c\in B\}\in\V.$$
Then apply Theorem \ref{mainmain} and get the existence of a 1-1 sequence $(a_n)_{n\in\N}$ in $S$ such that 
$\text{FE}_{\theta,N,\Phi}(a_n)_{n=1}^\infty\subseteq C_j\cap \Theta(B\times B)$.
For every $n$, the element $a_n\in \Theta(B\times B)$, and so we can pick $b_n,c_n\in B$ 
such that $a_n=b_n*_\theta c_n=\theta(b_n) c_n$.
Then the sequences $(b_n)_{n\in\N},(c_n)_{n\in\N}\subseteq B$ have the desired properties.
\end{proof}

\subsection{The special case of exponentiation on natural numbers}

The proof of Theorem \ref{mainmain} is quite technical and rich in detail. 
In fact, the generality of the result for arbitrary exponential-like operations 
requires complex notation that can distract the reader from the ideas behind the 
construction. 
The most relevant special case of our result is the exponentiation on natural numbers, 
which is obtained when $(S,\ast)=(\N,\,\cdot\,)$ and
$\theta$ is the identity. For this case it is possible to provide a simplified 
and ``more readable” proof that is presented below.

\begin{theorem}[Special case on exponentiation on $\N$ -- simplified version]\label{exp}
For every finite coloring of $\N$ there exists an increasing sequence
$(a_n)_{n\in\N}$ such that $\text{FE}(a_n)_{n=1}^\infty$ is monochromatic.
\end{theorem}

\begin{proof}
For sake of simplicity, we just construct the first three terms $a_1<a_2<a_3$ of the sequence,
and show that all elements of the corresponding set of exponentiations
\[a_1,\: a_2,\: a_3,\: a_2^{a_1},\: a_3^{a_1},\: a_3^{a_2},\: a_3^{a_2a_1},\: a_3^{a_2^{a_1}}, \:a_3^{a_2^{a_1}a_1}\]
are monochromatic. 
By inductively iterating the procedure, one can then obtain the increasing sequence $(a_n)_{n\in\N}$
such that all finite exponentiations in $\text{FE}(a_n)_{n=1}^\infty$ are monochromatic.

\smallskip
Pick any minimal idempotent ultrafilter $\U$ in $(\beta\N, \odot)$,
and let $\V\coloneq \Theta(\U\otimes\U)$ be the image ultrafilter
of the tensor product $\U\otimes\U$ under the exponential function
$\Theta(a, b)= b^a$.
For $B\subseteq\N$ and $n\in \N$, we denote:
\begin{itemize}
\item
$B/n=\{k\in\N\mid n\cdot k\in B\}$.
\item
$B'=\{n\in\N\mid B/n\in \U\}$.
\item 
$B_n=\{k\in \N\mid k^n\in B\}$.
\item 
$\widehat{B}=\{n\in \N\mid B_n\in\U\}$.
\end{itemize}

Since $\U$ is multiplicatively idempotent, one has $B\in \U\Leftrightarrow B'\in \U$. Note also 
that, by the definition of $\V=\Theta(\U\otimes\U)$, one has $B\in\V\Leftrightarrow\widehat{B}\in\U$.

In the sequel, we will repeatedly use the following facts:
\begin{itemize}
\item[($\dagger)$]
If $B\in\V$ and $n\in\widehat{B}\cap\widehat{B}'$, 
then the set $B_n\cap\widehat{B}/n\cap\widehat{B}'/n\in\U$.
\item[($\ddagger$)] 
If $B\in \U$ then for every $\ell\in \N$ there exists a pattern $b, L, bL, \ldots, bL^\ell\in B$.
\end{itemize}
To prove $(\dagger)$, observe that $B_n\in\U$ since $n\in\widehat{B}$,
and $\widehat{B}/n\in\U$ since $n\in\widehat{B}'$; besides, 
$\widehat{B}/n\in\U\Leftrightarrow(\widehat{B}/n)'=\widehat{B}'/n\in\U$.
Property $(\ddagger)$ follows from Corollary \ref{brauerforcentralsets},
since $\U$ is multiplicatively minimal and idempotent, and so $B$ is multiplicatively central.

Given a finite coloring $\N=C_1\cup\ldots\cup C_r$,
let $A:=C_i\in\V$ be the color that belongs to $\V$.
We are now ready for the construction.

\begin{itemize}
\item
Pick $b_0\in\widehat{A}\cap\widehat{A}'\in \U$, and let
$E_0:=A_{b_0}\cap \widehat{A}/b_0\cap \widehat{A}'/b_0$.
\end{itemize}
By $(\dagger)$ we have that $E_0\in\U$.
\begin{itemize}
\item 
Let $\ell_1:=1+b_0$. Since $E_0\in\U$, by $(\ddagger)$ we can pick a pattern:
\[b_1,\: L_1,\: b_1L_1, \ldots, \; b_1L_1^{\ell_1}\in E_0.\]
\item
Let $a_1\coloneq L_1^{b_0}$.
\end{itemize}
Note that $a_1\in A$ because $L_1\in A_{b_0}$.
\begin{itemize}
\item
Let $D_1:=\{b_0b_1L_1^i\mid i\leq \ell_1\}$ and let
$E_1:=\bigcap_{x\in D_1}(A_x\cap\widehat{A}/x\cap\widehat{A}'/x)$.
\end{itemize}
Note that $D_1\subseteq\widehat{A}\cap\widehat{A}'$ and so, by $(\dagger)$, the set $E_1\in\U$.
\begin{itemize}
\item 
Let $\ell_2:=1+\max D_1$. 
By $(\ddagger)$ we can pick a pattern
\[b_2, \;L_2, \;b_2L_2, \ldots,\:b_2L_2^{\ell_2 }\in E_1.\]
\item
Let $a_2\coloneq L_2^{b_0b_1L_1} $.
\end{itemize} 
Note that $a_2>a_1$.
Since $L_2\in\bigcap_{x\in D_1} A_x$, and $b_0 b_1 L_1, b_0b_1L_1^{1+b_0}\in D_1$,
we have that $a_2=L_2^{b_0b_1L_1}\in A$ and $a_2^{a_1}=L_2^{b_0b_1L_1^{1+b_0}}\in A$.
\begin{itemize}
\item
Let $D_2:=\{b_0b_1L_1^ib_2L_2^j\mid i\leq \ell_1,  j\leq \ell_2\}$
and let $E_2:=\bigcap_{x\in D_2}(A_x\cap\widehat{A}/x\cap\widehat{A}'/x)$.
\end{itemize}
Note that $D_2\subseteq\widehat{A}\cap\widehat{A}'$ and so, by $(\dagger)$, the set $E_2\in\U$.
\begin{itemize}
\item 
Let $\ell_3:=1+\max D_2$. 
By $(\ddagger)$ we can pick a pattern
\[b_3,\: L_3,\: b_3L_3, \ldots, \; b_3L_3^{\ell_3}\in E_2.\]
\item
Let $a_3:= L_3^{b_0b_1L_1b_2L_2}$.
\end{itemize}
Note that $a_3>a_2$.
Since $L_3\in\bigcap_{x\in D_2}A_x$, the elements below belong to $A$ 
because all the considered exponents belong to $D_2$:
\begin{align*}
a_3 & = L_3^{b_0b_1L_1b_2L_2}
\\
a_3^{a_1} & = L_3^{b_0b_1L_1^{1+b_0}b_2L_2}
\\ 
a_3^{a_2} & = L_3^{b_0b_1L_1b_2L_2^{1+b_0b_1L_1}}
\\
a_3^{a_2a_1} & = L_3^{b_0b_1L_1^{1+b_0}b_2L_2^{1+b_0b_1L_1}}
\\
a_3^{a_2^{a_1}} & = L_3^{b_0b_1L_1b_2L_2^{1+b_0b_1L_1^{1+b_0}}}
\\
a_3^{a_2^{a_1}a_1} & = L_3^{b_0b_1L_1^{1+b_0}b_2L_2^{1+b_0b_1L_1^{1+b_0}}}
\end{align*}
\end{proof}

\subsection{The general case}

We finally give a proof of Theorem \ref{mainmain} in its full generality.

For simplicity, we will adopt additive notation and use the symbol $+$ 
to indicate the operation $*$ on $S$.\footnote
{~See Notation \ref{additivenotation} and the following observations.}
To avoid possible ambiguities, we will always use the symbol $\cdot$ when
considering products between natural numbers.
Recall that, since $\theta$ is a homomorphism, we have that
$\theta(a)(\theta(b) c)=(\theta(a)\cdot\theta(b))c=\theta(a+b)c$ for all $a,b,c\in S$.

Let $\U$ be a non-principal minimal idempotent ultrafilter in $(\beta S,\circledast)$, 
and let $\V:=\Theta(\U\otimes\U)$ be the image ultrafilter of 
the tensor product $\U\otimes\U$ under the function $\Theta(a,b)=a *_\theta b=\theta(a) b$. 

For $B\subseteq S$ and $s\in S$, we denote:

\begin{itemize}
\item 
$B-s=\{t\in S\mid s+t\in B\}$.
\item 
$B'=\{s\in S\mid B-s\in \U\}$.
\item 
$B_s=\{t\in S\mid s*_\theta t\in B\}=\{t\in S\mid \theta(s)\, t\in B\}$.
\item 
$\widehat{B}=\{s\in S\mid B_s\in \U\}$.
\end{itemize}

Note that:
\begin{itemize}
\item
$B\in \U\Leftrightarrow B'\in \U$, since $\U$ is idempotent.
\item
$B\in\V\Leftrightarrow\widehat{B}\in\U$, by the definition of $\V=\Theta(\U\otimes\U)$.
\end{itemize}

In the sequel, we will use the following facts:
\begin{itemize}
\item[($\dagger)$]
If $B\in\V$ and $s\in\widehat{B}'$, 
then the shift $(\widehat{B}\cap\widehat{B}')-s=(\widehat{B}-s)\cap(\widehat{B}'-s)\in\U$.
\item[($\ddagger$)] 
If $B\in \U$ then for every $\ell\in \N$ there exists a pattern $b, L, b+L, \ldots, b+\ell L\in B$.
\end{itemize}
To show $(\dagger)$, note that
$\widehat{B}-s\in\U$ since $s\in\widehat{B}'$; besides, 
$\widehat{B}-s\in\U\Leftrightarrow(\widehat{B}-s)'=\widehat{B}'-s\in\U$.
Property $(\ddagger)$ follows from Corollary \ref{brauerforcentralsets},
since $\U$ is minimal and idempotent in $(\beta S, \circledast)$, and so $B$ is central.

We now proceed by induction along the lines of what done above 
for the special case of the exponential over $\N$.

\begin{itemize}
\item 
Pick $b_0\in\widehat{A}\cap\widehat{A}'\in \U$, and let $D_1:=\{b_0\}$ and
$E_1:= A_{b_0}\cap (\widehat{A}-b_0)\cap (\widehat{A}'-b_0)$.
\end{itemize}
Since $b_0\in\widehat{A}$, the set $A_{b_0}\in\U$;
and since $b_0\in\widehat{A}'$, by ($\dagger$) we have that 
$(\widehat{A}-b_0)\cap (\widehat{A}'-b_0)\in\U$. This shows that $E_1\in\U$.
\begin{itemize}
\item 
Let $\ell_1:=1+N\cdot\theta(b_0)$. By ($\ddagger$), 
we can find a pattern
$$b_1, L_1, b_1+L_1, \ldots, b_1+\ell_1 L_1\in E_1.$$
\item
Let $a_1:=\theta(b_0) L_1$.
\end{itemize}
Note that $a_1\in A$ because $L_1\in E_1\subseteq A_{b_0}$. 

Let us now turn to the inductive step $k+1$. Assume that
numbers $\ell_1,\ldots,\ell_k\in\N$, elements $b_0,\ldots,b_k$, $L_1,\ldots,L_k$, $a_1,\ldots,a_k\in S$, and sets
$D_1,\ldots,D_k, E_1,\ldots,E_k\subseteq S$ have been defined in such a way that
\begin{enumerate}
\item
$\ell_i=1+f_i(a_1, \ldots, a_{i-1})\cdot\theta(b_0+\sum_{t=1}^{i-1}(b_t+L_t))$ for $i=2,\ldots,k$.
\item
$D_i=\{b_0+\sum_{t=1}^{i-1}
(b_t+j_t L_t)\mid j_t\leq \ell_t\}\subseteq\widehat{A}\cap\widehat{A}'$ for $i=1,\ldots,k$.
\item
$E_i=\bigcap_{x\in D_i}(A_x\cap (\widehat A-x)\cap (\widehat A'-x))\in \U$ for $i=1,\ldots,k$.
\item
$b_i, L_i, b_i+L_i, \ldots , b_i+\ell_i L_i\in E_i$ for $i=1,\ldots,k$.
\item
$a_i=\theta(b_0+\sum_{t=1}^{i-1}(b_t+L_t))L_i$ for $i=1,2,\ldots,k$.
\item
$a_i\ne a_t$ for all $1\le t<i\le k$.
\item
$\text{EXP}_{\theta,N,\Phi}(a_1,\ldots,a_i)\subseteq A$ for $i=1,\ldots,k$.
\end{enumerate}

Then set:
\begin{itemize}
\item
$\ell_{k+1}:=1+f_{k+1}(a_1, \ldots, a_k)\cdot\theta(b_0+\sum_{t=1}^{k}(b_t+L_t))$.
\item
$D_{k+1}:=\{b_0+\sum_{t=1}^{k}
(b_t+j_t L_t)\mid j_t\leq \ell_t\}$.
\item
$E_{k+1}=\bigcap_{x\in D_{k+1}}(A_x\cap (\widehat A-x)\cap (\widehat A'-x))$.
\end{itemize}
Let us show that $D_{k+1}\subseteq \widehat{A}\cap\widehat{A}'$.
Given $j_t\le\ell_t$ for $t=1,\ldots,k$, observe that $z:=b_0+\sum_{t=1}^{k-1} j_t L_t\in D_k$.
Besides, $b_k+j_kL_k\in E_k\subseteq (\widehat{A}-z)\cap(\widehat{A}'-z)$,
and so $z+b_k+j_k L_k=b_0+\sum_{t=1}^k(b_t+ j_t L_t)\in\widehat{A}\cap\widehat{A}'$.

Now let $x\in D_{k+1}$. Since $x\in\widehat{A}$, the set $A_x\in\U$; and since 
$x\in\widehat{A}'$, by $(\dagger)$ the set $(\widehat{A}-x)\cap(\widehat{A}'-x)\in\U$;
consequently the set $E_{k+1}\in\U$, as a finite intersection of elements of $\U$.

Let $\mu_k:=\theta(b_0+\sum_{t=1}^k(b_t+L_t))$.
By the root-finiteness of $S$, for every $i\le k$ the set 
$\{s\in S\mid \mu_k s=a_i\}$ is finite and so, since $\U$ is non-principal, the set 
$$\Gamma_k:=\{s\in S\mid \mu_k s\notin\{a_1,\ldots,a_k\}\}\in\U.$$
Then $E_{k+1}\cap\Gamma_k\in\U$ and by ($\ddagger$) we can find a pattern
$$b_{k+1}, L_{k+1}, b_{k+1}+L_{k+1}, \ldots, b_{k+1}+\ell_{k+1} L_{k+1}\in E_{k+1}\cap\Gamma_k.$$
Finally, define:
\begin{itemize}
\item
$a_{k+1}:=\mu_k L_{k+1}$.
\end{itemize}

Note that $L_{k+1}\in\Gamma_k$ implies that $a_{k+1}\ne a_i$ for all $i=1,\ldots,k$.

We are left to show that
$\text{EXP}_{\theta,N,\Phi}(a_1,\ldots,a_k,a_{k+1})\subseteq A\setminus\{a_1,\ldots,a_k\}$.
Recall that every element in that set of $\theta$-exponentiations 
can be written in the form $(\sum_{t=1}^{k}\lambda_t a_t)*_\theta a_{k+1}$ for appropriate
$0\le\lambda_1\le N$, and $0\le \lambda_i\le f_i(a_1,\ldots,a_{i-1})$ for $2\le i\le k$
(see Definition \ref{def-FEthetaNPhi}).
Since $\theta$ is a homomorphism, we have the following equalities:
\begin{multline*}
\left(\sum_{i=1}^k\lambda_i a_i\right)*_\theta\,a_{k+1} =
\theta \left(\sum_{i=1}^k\lambda_i a_i\right) a_{k+1} =
\\
= \theta\left(\sum_{i=1}^k\lambda_i a_i\right) \left(\theta\left(b_0+\sum_{i=1}^k(b_i+L_i)\right) L_{k+1}\right) =
\\
= \left(\theta\left(\sum_{i=1}^k\lambda_i a_i\right)\cdot\ \theta\left(b_0+\sum_{i=1}^k(b_i+L_i)\right)\right) L_{k+1} = 
\theta(x) L_{k+1}
\end{multline*}
where
\begin{multline*}
x:=b_0+\sum_{i=1}^k(b_i+L_i)+\sum_{i=1}^k\lambda_i a_i =
b_0+\sum_{i=1}^k(b_i+L_i)+\sum_{i=1}^k\lambda_i\left(\theta\left(b_0+\sum_{j=1}^{i-1}(b_j+L_j)\right)L_i\right)=
\\
=b_0+\sum_{i=1}^k\left(b_i+\left(1+\lambda_i\cdot \theta\left(b_0+\sum_{j=1}^{i-1}(b_j+L_j)\right)\right)L_i\right).
\end{multline*}
Now observe that 
$1+\lambda_i\cdot \theta\left(b_0+\sum_{j=1}^{i-1}(b_j+L_j)\right)\le\ell_i$ for every $i=1,\ldots,k$, 
and so $x\in D_{k+1}$.
Since $L_{k+1}\in E_{k+1}\subseteq\bigcap_{x\in D_{k+1}}A_x$, it follows that 
$\theta(x) L_{k+1}\in A$, as desired.

\begin{remark}\label{braueridempotent}
The previous proof shows that the assumptions about the ultrafilter $\U$ 
in the statement of Theorem \ref{mainmain} could be weakened. 
In fact, any non-principal ultrafilter $\U$ that satisfies the following property is sufficient 
for the previous arguments to hold.
\begin{itemize}
\item[$(\star)$]
\emph{$\U$ is idempotent and for every $B\in\U$ and for every $\ell\in\N$, 
there exist elements $a, d, a+d, \ldots, a+\ell d\in B$.}
\end{itemize}
As noted in Corollary \ref{brauerforcentralsets}, every minimal idempotent ultrafilter 
satisfies $(\star)$. However, the converse is not true. Indeed, consider the set 
$\Delta^*\subset\beta\N$ of those ultrafilters that contain only sets of positive upper 
Banach density. Since $\Delta^*$ is a closed sub-semigroup (actually a two-sided ideal) 
of $(\beta\N,\oplus)$, there exist idempotent ultrafilters in $\Delta^*$, called \emph{essential idempotent}.
Sets belonging to an essential idempotent ultrafilter are called \emph{$D$-sets} 
and have relevant combinatorial properties, similar to those of central sets. 
It was shown by M. Beiglb\"ock, V. Bergelson, T. Downarowicz, and A. Fish in \cite{bbdf} 
that all Rado systems are solvable in $D$-sets, and hence
they contain ``Brauer configurations" $a, d, a+d, \ldots, a+\ell d$ for every $\ell\in\N$. 
So, every essential idempotent ultrafilter satisfies property $(\star)$; 
however, it is known that there exist essential idempotent ultrafilters 
that are not minimal (see \cite[\S 1.14]{bm}).
\end{remark}

\medskip
\section{Open questions}

We close this paper by itemizing a few open questions
that naturally arise from our research.

\smallskip
\begin{enumerate}
\item
Are there other examples of non-associative operations
that satisfy the infinitary partition regularity property of Hindman's Theorem?
\emph{E.g.}, can one extend the Main Theorem \ref{main} so to also include
a larger class of non-associative operations?

\smallskip
\item
More generally, is there a simple characterization of those non-associative operations
that satisfy Hindman's Theorem? Or at least,
is there a simple characterization of those operations
$F:S\times S\to S$ such that the triples of the form
$a, b, F(a,b)$ are monochromatic in $S$?

\smallskip
\item
Are there non-principal idempotent ultrafilters on $\N$ with respect to the exponential
function $\text{Exp}:(a,b)\mapsto b^a$?
(That is, ultrafilters $\U$ on $\N$ such that the image
ultrafilter $\text{Exp}(\U\otimes\U)=\U$.)
In \cite{lb} it was proved that the answer is negative if we 
reverse the order of exponentiation, \emph{i.e.}, if we replace
$\text{Exp}$ with $\text{Exp}':(a,b)\mapsto a^b$.

\smallskip
\item
Under the hypotheses of the Main Theorem \ref{main},
one could also consider finite patterns of products in the semigroup $(S,*)$:
$$\text{FP}(a_1,\ldots,a_\ell)=\{a_{n_1}*\ldots*a_{n_k}\mid 1\le n_1<\ldots<n_k\le\ell\}.$$
Given $\ell$, is it possible to extend that result to have
also $\text{FP}(a_1,\ldots,a_\ell)$ in the same color as $\text{FE}_{\theta,N,\Phi}(a_n)_{n=1}^\infty$?
(At least when the considered exponential-like operation is the usual exponentiation 
between natural numbers.) 
In this way, one would obtain a complete generalization
of Sahasrabudhe's results in \cite{sa}.

\smallskip
\item
Let us call \emph{Brauer} an ultrafilter $\U$ on $\N$ that is a 
witness to Brauer's Theorem, \emph{i.e.}, such that for every $B\in\U$ 
and for every $\ell\in\N$, there exist elements $a, d, a+d, \ldots, a+\ell d\in B$.
As discussed in the Remark \ref{braueridempotent}, every
essential idempotent ultrafilter is Brauer.
Are there idempotent Brauer ultrafilters that are not essentially idempotent?
\end{enumerate}

\bigskip
\noindent
\textbf{Acknowledgement.}
The authors express their gratitude to the anonymous referee for his/her attentive reading of the manuscript,
and for his/her constructive comments that were useful to simplify our proof.
These were especially useful for the inclusion of a simpler proof
in case of exponentiation between natural numbers.

\bigskip
\noindent
\textbf{Funding.}\
M. Di Nasso was supported by the Italian research project PRIN 2022
``Logical methods in combinatorics", 2022BXH4R5, 
Italian Ministry of University and Research (MUR), and is a member of the INdAM research group GNSAGA.
M. Ragosta is supported by project 25-15571S of the Czech Science Foundation (GA\v{C}R);
this work has been supported by Charles University Research Centre programme No.UNCE/24/SSH/026.
We acknowledge the MUR Excellence Department Project awarded to the Department of Mathematics, 
University of Pisa, CUP I57G22000700001.

\end{document}